\documentclass[11pt]{amsart}
\usepackage{amssymb}
\usepackage{graphicx}
\usepackage{enumerate}
\usepackage{mathtools}
\usepackage{xcolor,soul}
\usepackage{cite}
\usepackage{tikz,tikz-cd}
\usetikzlibrary{matrix}

\usepackage[breaklinks=true]{hyperref}\hypersetup{colorlinks,
  linkcolor={green!50!black},
  citecolor={green!50!black},
  urlcolor=blue
}

\usepackage{caption}

\usepackage{amsmath, amsthm}
 \usepackage{relsize}

\definecolor{lgray}{gray}{0.82}
\definecolor{rredd}{rgb}{120,0,0}

\usepackage[margin=1.2in]{geometry}

\newtheorem{thm}{Theorem}[section]
\newtheorem{prop}[thm]{Proposition}
\newtheorem{lem}[thm]{Lemma}
\newtheorem{cor}[thm]{Corollary}

\newcommand{\theoremname}{Theorem:}

\usepackage{bbm}

\newcommand{\R}{\mathbb{R}}

\newcommand{\g}{\mathfrak{g}}
\newcommand{\wTFe}{\widetilde{\mathit w\!T\!F}}

\newcommand{\calA}{\mathcal{A}^{sw}}
\newcommand{\up}{\Upsilon}
\newcommand{\id}{\operatorname{id}}
\newcommand{\D}{\mathcal{D}}

\def\arXiv#1{{\href{http://front.math.ucdavis.edu/#1}{arXiv:\linebreak[0]#1}}}

\newcommand{\rcross}{
\begin{tikzpicture}[scale=.2]
\draw[thick, ->](1,-1)--(-1.1,1.1);
\draw[ ultra thick](-1,-1)--(-.2,-.2);
\draw[ultra thick] (.2,.2)--(.95,.95);
\draw[ thick,->] (.2,.2)--(1.1,1.1);
\draw[line width= 2.3 mm, white](.4,-.4)--(-.4,.4);
\draw[ultra thick](1,-1)--(-.95,.95);
\end{tikzpicture}}

\newcommand{\mrcross}{
\begin{tikzpicture}[scale=.2]

\draw[thick, ->](1,-1)--(-1.1,1.1);
\draw[red](-1,-1)--(-.2,-.2);
\draw[red,->] (.2,.2)--(.95,.95);

\draw[line width= 2.3 mm, white](.4,-.4)--(-.4,.4);
\draw[ultra thick](1,-1)--(-.95,.95);
\end{tikzpicture}}

\newcommand{\lcross}{
\begin{tikzpicture}[scale=.2]

\draw[ thick,->](-1,-1)--(1.1,1.1);
\draw[ultra thick](1,-1)--(.2,-.2);
\draw[ultra thick] (-.2,.2)--(-.95,.95);
\draw[ thick,->] (-.2,.2)--(-1.1,1.1);

\draw[line width= 2.3 mm, white](-.4,-.4)--(.4,.4);
\draw[ ultra thick](-1,-1)--(.95,.95);
\end{tikzpicture}}

\newcommand{\mlcross}{
\begin{tikzpicture}[scale=.2]

\draw[ thick,->](-1,-1)--(1.1,1.1);
\draw[red](1,-1)--(.2,-.2);
\draw[red,->] (-.2,.2)--(-1,1);

\draw[line width= 2.3 mm, white](-.4,-.4)--(.4,.4);
\draw[ultra thick,](-1,-1)--(.95,.95);
\end{tikzpicture}}

\newcommand{\capoff}{
\begin{tikzpicture}[scale=.2]
\draw[ultra thick](0,-1)--(0,1);
\filldraw[fill=black, draw=black] (0,1) circle (3mm) ;
\end{tikzpicture}}

\newcommand{\sphere}{
\begin{tikzpicture}[scale=.1]
\draw[](0,-.5)--(0,1);
\filldraw[fill=black, draw=black] (0,-.5) circle (2mm) ;
\filldraw[fill=black, draw=black] (0,1) circle (2mm) ;
\end{tikzpicture}}

\newcommand{\calAt}{
\calA ( \hspace{.05cm}\begin{tikzpicture}[scale=.2, baseline=.0mm  ]
\draw[](0,-.5)--(0,1);
\filldraw[fill=black, draw=black] (0,-.5) circle (2mm) ;
\filldraw[fill=black, draw=black] (0,1) circle (2mm) ;
\node[right] at (-.4,1) {\footnotesize $k_1$}; \end{tikzpicture}  
 \hspace{-.1cm} \begin{tikzpicture}[scale=.2,baseline=.0mm]
\draw[red,->](.5,-.5)--(.5,1.2);
\node[right] at (.15,1) {\footnotesize $k_2$};\end{tikzpicture}\hspace{-.1cm})
}
\newcommand{\calAtSph}{
\calA ( \hspace{.05cm}\begin{tikzpicture}[scale=.2, baseline=.0mm  ]
\draw[](0,-.5)--(0,1);
\filldraw[fill=black, draw=black] (0,-.5) circle (2mm) ;
\filldraw[fill=black, draw=black] (0,1) circle (2mm) ;
\node[right] at (-.4,1) {\footnotesize $1$}; \end{tikzpicture}  
 \hspace{-.1cm} \begin{tikzpicture}[scale=.2,baseline=.0mm]
\draw[red,->](.5,-.5)--(.5,1.2);
\node[right] at (.35,1) {\footnotesize $0$};\end{tikzpicture}\hspace{-.1cm})
}

\newcommand{\calAtFour}{
\calA ( \hspace{.05cm}\begin{tikzpicture}[scale=.2, baseline=.0mm  ]
\draw[](0,-.5)--(0,1);
\filldraw[fill=black, draw=black] (0,-.5) circle (2mm) ;
\filldraw[fill=black, draw=black] (0,1) circle (2mm) ;
\node[right] at (-.4,1) {\footnotesize $2$}; \end{tikzpicture}  
 \hspace{-.1cm} \begin{tikzpicture}[scale=.2,baseline=.0mm]
\draw[red,->](.5,-.5)--(.5,1.2);
\node[right] at (.25,1) {\footnotesize $2$};\end{tikzpicture}\hspace{-.1cm})
}

\newcommand{\calAone}{
\calA ( \begin{tikzpicture}[scale=.2, baseline=.0mm  ]
\draw[](0,-.5)--(0,1);
\filldraw[fill=black, draw=black] (0,-.5) circle (2mm) ;
\filldraw[fill=black, draw=black] (0,1) circle (2mm) ;
\draw[red,->](.75,-.5)--(.75,1.2);
\end{tikzpicture})
}

\newcommand{\calAtwo}{
\calA ( \begin{tikzpicture}[scale=.2, baseline=.0mm  ]
\draw[ultra thick,](0,-.5)--(0,1.2);
\draw[ultra thick,](1,-.5)--(1,1.2);
\draw[thick,->](0,-.5)--(0,1.3);
\draw[thick,->](1,-.5)--(1,1.3);
\end{tikzpicture})
}

\newcommand{\mlambda}{
\begin{tikzpicture}[scale=.2]
\draw[red](-.09,-.09)--(-1,-1);
\draw[ thick,->](0,-.08)--(0,1.1);
\draw[ultra thick](1, -1)--(0,0)--(0,.93);

\end{tikzpicture}}

\newcommand{\bulambda}{
\begin{tikzpicture}[scale=.2,baseline=-.1cm]
\draw[ultra thick,](0,-1)--(0,0)--(-.95, .95);
\draw[thick,->](0,0)--(-1.1, 1.1);
\draw[ultra thick, ](.5,.5)--(.95,.95);
\draw[ thick, ->](.5,.5)--(1.1,1.1);
\end{tikzpicture}}

\newcommand{\bdlambda}{
\begin{tikzpicture}[scale=.2,baseline=-.1cm]
\draw[thick, ->](0,-.09)--(0,1.1);
\draw[ultra thick](1, -1)--(.5,-.5);
\draw[ultra thick](0,.9)--(0,0)--(-1,-1);
\end{tikzpicture}}

\newcommand{\rlambda}{
\begin{tikzpicture}[scale=.2]
\draw[red,->](0,0)--(0,1);
\draw[red](1, -1)--(0,0);
\draw[red](0,0)--(-1,-1);
\end{tikzpicture}}

\newcommand{\threadedsphere}{
\begin{tikzpicture}[scale=.2]
\draw[thick](0,-1)--(0,1);
\draw[red](-1,0)--(-.4,0);
\draw[red] (.4,0)--(1,0);
\filldraw[fill=black, draw=black] (0,1) circle (2mm) ;
\filldraw[fill=black, draw=black] (0,-1) circle (2mm) ;
\end{tikzpicture}
}

\newcommand{\threadedspheretube}{
\begin{tikzpicture}[scale=.2]
\draw[] (0,0) circle (1cm);
\begin{scope}
    \clip (-1,0) rectangle (1,1);
\draw[ dotted] (0,0) ellipse (1cm and .3cm);
\end{scope}
\begin{scope}
    \clip (-.64,0) rectangle (1,-1);
    \draw[] (0,0) ellipse (1cm and .3cm);
\end{scope}
\begin{scope}
    \clip (-1,1) rectangle (1,-1);
    \draw[] (-.8,0) ellipse (.2cm and .45cm);
\end{scope}

\draw [white,fill=white]  (-1.05,.15) rectangle (-.97,-.15);
\draw[red, ] (-2,0)--(-.63,0)(-.57,0)--(.7,0)(1.2,0)--(2,0);
 \end{tikzpicture}
 }

    \newcommand{\deltasphere}{
  \node[] at (0,0) {\huge$\Delta$};

  \begin{scope}[scale=.4,xshift=3.5cm, yshift=-2cm]
   \draw[] (0,0) circle (1cm);
\begin{scope}
    \clip (-1,0) rectangle (1,1);
\draw[ dashed] (0,0) ellipse (1cm and .3cm);
\end{scope}
\begin{scope}
    \clip (-.64,0) rectangle (1,-1);
    \draw[] (0,0) ellipse (1cm and .3cm);
\end{scope}
\begin{scope}
    \clip (-1,1) rectangle (1,-1);
    \draw[] (-.8,0) ellipse (.2cm and .45cm);
\end{scope}

\end{scope}

  }

    \newcommand{\inlinedeltasphere}{
  \node[] at (0,0) {\Large$\Delta$};

  \begin{scope}[scale=.3,xshift=3.5cm, yshift=-2cm]
   \draw[] (0,0) circle (1cm);
\begin{scope}
    \clip (-1,0) rectangle (1,1);
\draw[ dashed] (0,0) ellipse (1cm and .3cm);
\end{scope}
\begin{scope}
    \clip (-.64,0) rectangle (1,-1);
    \draw[] (0,0) ellipse (1cm and .3cm);
\end{scope}
\begin{scope}
    \clip (-1,1) rectangle (1,-1);
    \draw[] (-.8,0) ellipse (.2cm and .45cm);
\end{scope}

\end{scope}

  }

  \newcommand{\pictureone}{
  \begin{tikzpicture}[baseline=-1.3mm]
  \draw[]  (0,0) ellipse (.1cm and .25cm);
  \draw[white, line width =2.3 mm] (-.25,0)--(0,0);
    \draw[] (-.25,0)--(0.05,.0)(.15,0)--(.3,0);
    \node[] at(.55,0) {\#};

    \begin{scope}[xshift=1.05cm]
      \draw[]  (0,0) ellipse (.1cm and .25cm);
  \draw[white, line width =2.3 mm] (-.25,0)--(0,0);
    \draw[] (-.25,0)--(0.05,.0)(.15,0)--(.3,0);
     \node[] at(.65,0) {=};
\end{scope}

 \begin{scope}[xshift=2.3cm]
   \draw[]  (0,0) ellipse (.08cm and .23cm);
          \draw[]  (0,0) ellipse (.15cm and .3cm);
  \draw[white, line width =2.3 mm] (-.25,0)--(0,0);
    \draw[] (-.25,0)--(0.05,.0)(.2,0)--(.35,0);
     \node[] at(.7,0) {=};
\end{scope}

 \begin{scope}[xshift=4cm]
 \node[] at (-.55,0) {\Large $\Delta$};
  \draw[]  (-.27,-.22) ellipse (.04 cm and .09cm);
 \draw[]  (0,0) ellipse (.1cm and .25cm);
  \draw[white, line width =2.3 mm] (-.25,0)--(0,0);
    \draw[] (-.25,0)--(0.05,.0)(.15,0)--(.3,0);

\end{scope}\end{tikzpicture}}

   \newcommand{\picturetwo}{
   \begin{tikzpicture}[baseline=-1.3mm, scale=.25]
\draw[] (0,0) circle (1cm);
\begin{scope}
    \clip (-1,0) rectangle (1,1);
\draw[ dotted] (0,0) ellipse (1cm and .3cm);
\end{scope}
\begin{scope}
    \clip (-.64,0) rectangle (1,-1);
    \draw[] (0,0) ellipse (1cm and .3cm);
\end{scope}
\begin{scope}
    \clip (-1,1) rectangle (1,-1);
    \draw[] (-.8,0) ellipse (.2cm and .45cm);
\end{scope}

\draw [white,fill=white]  (-1.05,.15) rectangle (-.97,-.15);
\draw[red,  ] (-2,0)--(-.63,0)(-.57,0)--(.7,0)(1.2,0)--(2,0);

\node[] at (3,0){\large $\#$};

\begin{scope}[xshift=6cm]
\draw[] (0,0) circle (1cm);
\begin{scope}
    \clip (-1,0) rectangle (1,1);
\draw[ dotted] (0,0) ellipse (1cm and .3cm);
\end{scope}
\begin{scope}
    \clip (-.64,0) rectangle (1,-1);
    \draw[] (0,0) ellipse (1cm and .3cm);
\end{scope}
\begin{scope}
    \clip (-1,1) rectangle (1,-1);
    \draw[] (-.8,0) ellipse (.2cm and .45cm);
\end{scope}

\draw [white,fill=white]  (-1.05,.15) rectangle (-.97,-.15);
\draw[red,  ] (-2,0)--(-.63,0)(-.57,0)--(.7,0)(1.2,0)--(2,0);
\end{scope}

\node[] at (10,0){\large $=$};

\begin{scope}[xshift=16cm]

 \begin{scope}[xshift=-4cm, scale=.8]
\node[] at (0,0) {\Large $\Delta$};

  \begin{scope}[scale=.5,xshift=3.5cm, yshift=-2cm]
   \draw[] (0,0) circle (1cm);
\begin{scope}
    \clip (-1,0) rectangle (1,1);
\draw[ dashed] (0,0) ellipse (1cm and .3cm);
\end{scope}
\begin{scope}
    \clip (-.64,0) rectangle (1,-1);
    \draw[] (0,0) ellipse (1cm and .3cm);
\end{scope}
\begin{scope}
    \clip (-1,1) rectangle (1,-1);
    \draw[] (-.8,0) ellipse (.2cm and .45cm);
\end{scope}

\end{scope}

\end{scope}

\draw[] (0,0) circle (1cm);
\begin{scope}
    \clip (-1,0) rectangle (1,1);
\draw[ dotted] (0,0) ellipse (1cm and .3cm);
\end{scope}
\begin{scope}
    \clip (-.64,0) rectangle (1,-1);
    \draw[] (0,0) ellipse (1cm and .3cm);
\end{scope}
\begin{scope}
    \clip (-1,1) rectangle (1,-1);
    \draw[] (-.8,0) ellipse (.2cm and .45cm);
\end{scope}

\draw [white,fill=white]  (-1.05,.15) rectangle (-.97,-.15);
\draw[red,  ] (-2,0)--(-.63,0)(-.57,0)--(.7,0)(1.2,0)--(2,0);

\end{scope}
\end{tikzpicture} }

\def\aft{{\overrightarrow{4T}}}
\def\aAS{{\overrightarrow{AS}}}
\def\aSTU{{\overrightarrow{STU}}}
\def\aIHX{{\overrightarrow{IHX}}}
\def\aVI{{\overrightarrow{VI}}}

\usetikzlibrary{arrows}
\usepackage{wrapfig}

\theoremstyle{definition}

\theoremstyle{remark}

\numberwithin{equation}{section}

\usepackage[T1]{fontenc}

\title{ Ribbon 2-knots, 1+1=2, and Duflo's Theorem for arbitrary Lie Algebras}

\begin{document}
\date{\today}
\title{{Ribbon 2-knots, 1+1=2, and Duflo's Theorem for arbitrary Lie Algebras}}
\author{Dror~Bar-Natan}
\address{
  Department of Mathematics\\
  University of Toronto\\
  Toronto Ontario M5S 2E4\\
  Canada
}
\email{drorbn@math.toronto.edu}
\urladdr{http://www.math.toronto.edu/~drorbn}

\author{Zsuzsanna Dancso}
\address{
  School of Mathematics and Statistics\\
  The University of Sydney\\
  Eastern Ave\\
  Camperdown NSW 2006, Australia
}
\email{zsuzsanna.dancso@sydney.edu.au}
\urladdr{http://www.zsuzsannadancso.net}

\author{Nancy Scherich}
\address{
Department of Mathematics\\
South Hall, Room 6607\\
University of California\\
Santa Barbara, CA 93106-3080\\
United States
}
\email{nscherich@math.ucsb.edu}
\urladdr{http://www.nancyscherich.com}


\thanks{This work was partially supported by NSERC grant RGPIN 262178 and by ARC DECRA DE170101128.}

\begin{abstract}
We explain a direct topological proof for the multiplicativity of Duflo isomorphism for arbitrary finite dimensional Lie algebras, and derive the explicit formula for the Duflo map. The proof follows a series of implications, starting with ``the calculation 1+1=2 on a 4D abacus'', using the study of {\em homomorphic expansions} (aka universal finite type invariants) for ribbon 2-knots, and the relationship between the corresponding associated graded space of {\em arrow diagrams} and universal enveloping algebras. This complements the results of the first author, Le and Thurston, where similar arguments using a ``3D abacus'' and the Kontsevich Integral were used to deduce Duflo's theorem for {\em metrized} Lie algebras; and results of the first two authors on finite type invariants of w-knotted objects, which also imply a relation of 2-knots with the Duflo theorem in full generality, though via a lengthier path.
\end{abstract}
\maketitle

\section{Introduction}

\subsection{Executive summary for experts} In~\cite{BLT}, the first author, Thang Le and Dylan Thurston tell a story of how a certain topological equality \pictureone\ ($\Delta$~for ``doubling''), or ``$1+1=2$ as computed on an abacus'', leads via the Kontsevitch integral to an equality of {\em chord diagrams}, which, given a metrized Lie algebra $\g$, can be interpreted as an equality in (a completion of) $S(\g)_\g\otimes S(\g)_g\otimes U(\g)$, which can be interpreted as ``the multiplicative property of the Duflo isomorphism''.

However, chord diagrams only describe tensors related to metrized Lie algebras, while the Duflo isomorphism is multiplicative for {\em all} finite-dimensional Lie algebras. Hence, as told in~\cite{BLT}, the ``1+1=2'' story is less general than it could be.

This paper removes this blemish by raising ``1+1=2'' one dimension up to be an equality \picturetwo\ ($\Delta$~for doubling) of 2-knots in $\R^4$, which then, using an appropriate replacement of the Kontsevich integral, becomes an equality of {\em arrow diagrams}, which in itself can be interpreted as an equality in (a completion of) $S(\g^\ast)_\g\otimes S(\g^\ast)_g\otimes U(\g)$ for an arbitrary finite-dimensional Lie algebra $\g$, proving the multiplicative property of the Duflo isomorphism in full generality.

\subsection{Introduction for all}For a finite dimensional Lie algebra $\g$, the Duflo isomorphism is an algebra isomorphism $\D: S(\g)^\g \rightarrow U(\g)^\g$, where $U(\g)^\g$ and $S(\g)^\g$ are the $\g$ invariant subspaces for the adjoint action of $\g$ on the universal enveloping algebra and the symmetric algebra. (Recall $x$ is called invariant if $g\cdot x=0$ for every $g\in \g$.) The map $\D$ is given by an explicit formula. The difficulty is in showing that this formula represents a homomorphism, namely that it is multiplicative. We will henceforth refer to the problem of showing the multiplicativity of the Duflo map as {\em the Duflo problem}.

The Duflo isomorphism was first described for semi-simple Lie algebras by Harish-Chandra in 1951 \cite{HC}. Kirillov conjectured that a formulation of Harish-Chandra's map  was an algebra isomorphism for all finite dimensional Lie algebras. Duflo proved Kirillov's conjecture in 1977 \cite{Duflo}, and it is now referred to as Duflo's Theorem. Since then, there have been many proofs of Duflo's theorem using techniques outside the setting of the originally formulated problem. For metrized Lie algebras, a topological proof was found by the first author, Le and Thurston in 2009 \cite{BLT} using the Kontsevich integral and a knot theoretic interpretation of ``$1+1=2$ on an abacus''. In this paper we give a new topological proof of Duflo's theorem for {\em arbitrary finite dimensional Lie algebras} using a ``4-dimensional abacus'' instead of an ordinary 3-dimensional one.

The Dulfo problem is also implied by the now-proven Kashiwara--Vergne (KV) conjecture \cite{KV}. The KV conjecture states that a certain set of equations has a solution in the group of \textit{tangential automorphisms} of the degree completed free Lie algebra on 2 generators. One can extract the Duflo isomorphism from such a solution.  The KV conjecture was proven by \cite{AM} in 2006 using deformation quantization. New proofs exploiting the relationship between the KV equations and Drinfel'd associators were found by Alekseev, Torossian and Enriquez shortly thereafter \cite{AT, AET}. A topological context and solution in terms of the 4-dimensional knot theory of {\em w-foams} was established by the first two authors in \cite{WKO2, WKO3}. In this context, the KV-conjecture is equivalent to the existence of a \textit{homomorphic expansion} for w-foams. In this paper, we directly address how such a homomorphic expansion gives rise to a solution of the Duflo problem and a formula for $\D$, and thus completing a topological solution of the Duflo problem in full generality.

This paper is structured to follow the implications shown in the Figure \ref{sketchofproof}. We start with an intuitive topological statement ``$1+1=2$'' and interpret this in the setting of w-foams. Using the homomorphic expansion $Z$ and the tensor interpretation map $T$, we can re-interpret ``$1+1=2$'' as an equality in $\hat S(\g^*)_\g \otimes  \hat U(\g)$. This will imply that our formulation of the Duflo isomorphism is an algebra homomorphism.  The essential ingredient in this process is the homomorphic expansion $Z$ of \cite{WKO2, WKO3}. Finally, we derive the explicit formula for the Duflo map from $Z$.

This paper builds on the setup and results of \cite{WKO2}. While we provide brief reviews of concepts, we assume some familiarity with finite type invariants and virtual/welded knots. The reader who is new to the subject may wish to have a copy of \cite{WKO2} on hand for reference: throughout the paper we will refer to specific sections for background details.

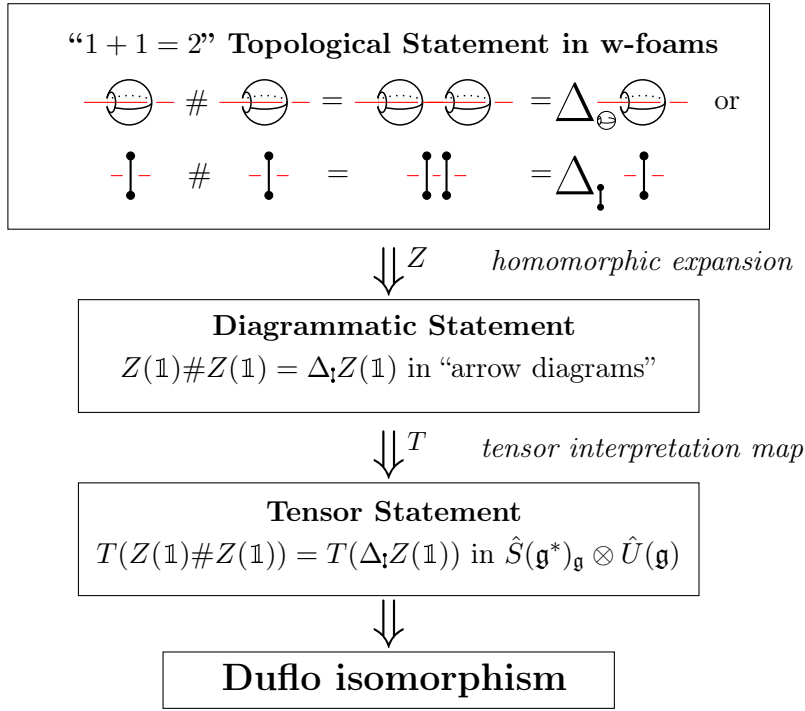
\begin{figure}

\begin{tikzpicture}[scale=.75]
\draw[] (-6.75,1.75) rectangle (6.75,-2.25);

\node[] at (0,1) {\textbf{  ``$1+1=2$'' Topological Statement in w-foams}};
\node[] at (6,0) {or};

 \node[] at (0,-3){\huge $\Downarrow$};
  \node[] at (.5,-2.75){$Z$};
   \node[] at (4.5, -2.85) {\textit{homomorphic expansion}};

   \draw[] (-5.5,-3.5) rectangle (5.5,-5.5);
\node[] at (0,-4) { \textbf{ Diagrammatic Statement}};
 \node[] at (0,-4.75) { $Z(\mathbbm{1}) \# Z(\mathbbm{1})=\Delta_{\sphere}Z( \mathbbm{1})$ in ``arrow diagrams''};

 \node[] at (0,-6.25){\huge $\Downarrow$};
 \node[] at (.5,-6){\small $T$};
 \node[] at (4.5, -6.18) { \textit{tensor interpretation map}};

 \draw[] (-5.5,-6.75) rectangle (5.5,-8.75);
\node[] at (0,-7.25) {\textbf{  Tensor Statement}};
\node[] at (0, -8.) { $T(Z(\mathbbm{1}) \# Z(\mathbbm{1}))=T(\Delta_{\sphere}Z( \mathbbm{1}))$   in $\hat S(\g^*)_\g \otimes  \hat U(\g)$};

 \draw[] (-4,-9.75) rectangle (4,-10.75);
 \node[] at (0,-9.25){\huge $\Downarrow$};
\node[]  at (0,-10.25) { \textbf{ {\Large Duflo isomorphism}}};

\begin{scope}[xshift=-4.6cm,yshift=0cm]
\begin{scope}[scale=.4]

\draw[line width= .2 mm] (0,0) circle (1cm);
\begin{scope}
    \clip (-1,0) rectangle (1,1);
\draw[line width= .2 mm, dotted] (0,0) ellipse (1cm and .3cm);
\end{scope}
\begin{scope}
    \clip (-.64,0) rectangle (1,-1);
    \draw[line width= .2 mm] (0,0) ellipse (1cm and .3cm);
\end{scope}
\begin{scope}
    \clip (-1,1) rectangle (1,-1);
    \draw[line width= .2 mm] (-.8,0) ellipse (.2cm and .45cm);
\end{scope}

\draw [white,fill=white]  (-1.05,.15) rectangle (-.97,-.15);
\draw[red,  ] (-2,0)--(-.63,0)(-.57,0)--(.8,0)(1.2,0)--(2,0);

\node[] at (3,0){ $\#$};

\begin{scope}[xshift=6cm]
\draw[line width= .2 mm] (0,0) circle (1cm);
\begin{scope}
    \clip (-1,0) rectangle (1,1);
\draw[line width= .2 mm, dotted] (0,0) ellipse (1cm and .3cm);
\end{scope}
\begin{scope}
    \clip (-.64,0) rectangle (1,-1);
    \draw[line width= .2 mm] (0,0) ellipse (1cm and .3cm);
\end{scope}
\begin{scope}
    \clip (-1,1) rectangle (1,-1);
    \draw[line width= .2 mm] (-.8,0) ellipse (.2cm and .45cm);
\end{scope}

\draw [white,fill=white]  (-1.05,.15) rectangle (-.97,-.15);
\draw[red,  ] (-2,0)--(-.63,0)(-.57,0)--(.8,0)(1.2,0)--(2,0);

\end{scope}

\node[] at (9,0){ $=$};

\begin{scope}[xshift=12cm]
\draw[line width= .2 mm] (0,0) circle (1cm);
\begin{scope}
    \clip (-1,0) rectangle (1,1);
\draw[line width= .2 mm, dotted] (0,0) ellipse (1cm and .3cm);
\end{scope}
\begin{scope}
    \clip (-.64,0) rectangle (1,-1);
    \draw[line width= .2 mm] (0,0) ellipse (1cm and .3cm);
\end{scope}
\begin{scope}
    \clip (-1,1) rectangle (1,-1);
    \draw[line width= .2 mm] (-.8,0) ellipse (.2cm and .45cm);
\end{scope}

\draw [white,fill=white]  (-1.05,.15) rectangle (-.97,-.15);
\draw[red,  ] (-2,0)--(-.63,0)(-.57,0)--(.8,0)(1.2,0)--(2,0);

\begin{scope}[xshift=3cm]
\draw[line width= .2 mm] (0,0) circle (1cm);
\begin{scope}
    \clip (-1,0) rectangle (1,1);
\draw[line width= .2 mm, dotted] (0,0) ellipse (1cm and .3cm);
\end{scope}
\begin{scope}
    \clip (-.64,0) rectangle (1,-1);
    \draw[line width= .2 mm] (0,0) ellipse (1cm and .3cm);
\end{scope}
\begin{scope}
    \clip (-1,1) rectangle (1,-1);
    \draw[line width= .2 mm] (-.8,0) ellipse (.2cm and .45cm);
\end{scope}

\draw [white,fill=white]  (-1.05,.15) rectangle (-.97,-.15);
\draw[red,  ] (-2,0)--(-.63,0)(-.57,0)--(.8,0)(1.2,0)--(2,0);

\end{scope}
\end{scope}
\end{scope}

\end{scope}

\begin{scope}[xshift=-4.5cm,yshift=0cm]
\begin{scope}[xshift=9cm,yshift=0cm, scale=.4]
\node[] at (-4.5,0) {$=$};
\begin{scope}[xshift=-3cm]\deltasphere\end{scope}

\draw[line width= .2 mm] (0,0) circle (1cm);
\begin{scope}
    \clip (-1,0) rectangle (1,1);
\draw[line width= .2 mm, dotted] (0,0) ellipse (1cm and .3cm);
\end{scope}
\begin{scope}
    \clip (-.64,0) rectangle (1,-1);
    \draw[line width= .2 mm] (0,0) ellipse (1cm and .3cm);
\end{scope}
\begin{scope}
    \clip (-1,1) rectangle (1,-1);
    \draw[line width= .2 mm] (-.8,0) ellipse (.2cm and .45cm);
\end{scope}

\draw [white,fill=white]  (-1.05,.15) rectangle (-.97,-.15);
\draw[red,  ] (-2,0)--(-.63,0)(-.57,0)--(.8,0)(1.2,0)--(2,0);
\end{scope}

\end{scope}

\begin{scope}[xshift=-4.05cm,yshift=-1.3cm]
\begin{scope}[scale=.35]

\begin{scope}[xshift=-1.5cm]
\draw[thick](0,-1)--(0,1);
\draw[red](-1,0)--(-.4,0);
\draw[red] (.4,0)--(1,0);
\filldraw[fill=black, draw=black] (0,1) circle (2mm) ;
\filldraw[fill=black, draw=black] (0,-1) circle (2mm) ;
\end{scope}
\node[] at (2,0){ $\#$};

\begin{scope}[xshift=5.5cm]
\draw[thick](0,-1)--(0,1);
\draw[red](-1,0)--(-.4,0);
\draw[red] (.4,0)--(1,0);
\filldraw[fill=black, draw=black] (0,1) circle (2mm) ;
\filldraw[fill=black, draw=black] (0,-1) circle (2mm) ;

\end{scope}

\node[] at (9,0){ $=$};

\begin{scope}[xshift=13.5cm]
\draw[thick](0,-1)--(0,1);
\draw[red](-1,0)--(-.4,0);
\draw[red] (.3,0)--(.7,0);
\filldraw[fill=black, draw=black] (0,1) circle (2mm) ;
\filldraw[fill=black, draw=black] (0,-1) circle (2mm) ;
\draw[thick](1,-1)--(1,1);
\draw[red](1.4,0)--(2,0);
\filldraw[fill=black, draw=black] (1,1) circle (2mm) ;
\filldraw[fill=black, draw=black] (1,-1) circle (2mm) ;
\end{scope}

\node[] at (19.3,0){ $=$};
\begin{scope}[xshift=21cm]
\node at (0,0) {\huge $\Delta$};
\begin{scope}[xshift=1.3cm, yshift=-1.1cm, scale=.5]
\draw[thick](0,-1)--(0,1);
\filldraw[fill=black, draw=black] (0,1) circle (3mm) ;
\filldraw[fill=black, draw=black] (0,-1) circle (3mm) ;
\end{scope}

\end{scope}

\begin{scope}[xshift=24.5cm]

\draw[thick](0,-1)--(0,1);
\draw[red](-1,0)--(-.4,0);
\draw[red] (.4,0)--(1,0);
\filldraw[fill=black, draw=black] (0,1) circle (2mm) ;
\filldraw[fill=black, draw=black] (0,-1) circle (2mm) ;
\end{scope}

\end{scope}
\end{scope}

\end{tikzpicture}
\captionof{figure}{The rough sketch of the proof.}\label{sketchofproof}

\end{figure}

\section{Understanding the Topological Statement and w-foams }\label{sec:wFoams}

\subsection{``4D Abacus Arithmetic'' }\label{sec:abacus} We begin by introducing the ``threaded sphere'' or ``abacus bead'' shown in Figure~\ref{fig:ThreadedSphere}: this is a knotted object in $\mathbb{R}^4$, and an element of the space of {\em w-foams} studied in \cite{WKO3}. To understand this 4D object, we describe it as a sequence of 3D slices, or ``frames of a 3D movie''.  The movie starts with two points $A$ and $B$. Point $B$ opens up to a circle,  $A$  flies through the circle, and $B$ closes to a point again.  In 4 dimensions this is a line threaded through a sphere with no intersections; and embedded pair. We depict this object as \threadedspheretube; this is a \textit{broken line/surface diagram} in the sense of \cite{CS}.

\begin{figure}
\begin{tikzpicture}[scale=.85]
\shade[ball color = gray!40, opacity = 0.3] (0,0) circle (.5cm);
\draw[thick] (0,-.5) ellipse (.1cm and .05cm);
\draw[red, fill=red] (-1,-1) circle (.05cm);

 \begin{scope}[xshift= 2.5cm]
  \shade[ball color = gray!40, opacity = 0.3] (0,0) circle (.5cm);
\draw[thick] (0,-.2) ellipse (.45cm and .1cm);
\draw[red,opacity = 0.5,->] (-1,-1)to [in=240, out=50] (-.1,-.4);
\draw[red, fill=red] (-.45,-.68) circle (.05cm);
\end{scope}

 \begin{scope}[xshift= 5cm]
 \shade[ball color = gray!40, opacity = .3] (0,0) circle (.5cm);
\draw[thick] (0,0) ellipse (.5cm and .1cm);
\draw[red,opacity = 0.5] (-1,-1)to [in=270, out=50] (0,-.2);
\draw[red,->,opacity = 0.5] (0,-.03)to [in=210, out=80] (.7,.5);
\draw[red, fill=red] (.2,.25) circle (.05cm);
\end{scope}

 \begin{scope}[xshift= 7.5cm]
 \shade[ball color = gray!40, opacity = 0.3] (0,0) circle (.5cm);
\draw[thick] (0,.5) ellipse (.1cm and .05cm);
\draw[red,opacity = 0.5] (-1,-1)to [in=270, out=50] (0,-.2);
\draw[red,opacity = 0.5] (0,0)to [in=210, out=90] (.7,.5);
\draw[red,opacity = 0.5] (.7,.5) to [in=250, out=30] (1,1);
\draw[red,opacity = 0.5] (.0,-.2) -- (0,0);
\draw[red, fill=red] (1,1) circle (.05cm);
\end{scope}

 \begin{scope}[xshift= -6cm, scale=.65]
\draw[thick] (0,0) circle (1cm);
\begin{scope}
    \clip (-1,0) rectangle (1,1);
\draw[ thick,dotted] (0,0) ellipse (1cm and .3cm);
\end{scope}
\begin{scope}
    \clip (-.64,0) rectangle (1,-1);
    \draw[thick] (0,0) ellipse (1cm and .3cm);
\end{scope}
\begin{scope}
    \clip (-1,1) rectangle (1,-1);
    \draw[thick] (-.8,0) ellipse (.2cm and .45cm);
\end{scope}

\draw [white,fill=white]  (-1.05,.15) rectangle (-.97,-.15);
\draw[red, thick ] (-2,0)--(-.63,0)(-.57,0)--(.8,0)(1.2,0)--(2,0);
\end{scope}

  \node[] at (-3.5,0){\huge $:=$};

 \end{tikzpicture}
 \captionof{figure}{The threaded sphere as a movie of a circle and a point in $\R^3$.}\label{fig:ThreadedSphere}
 \end{figure}
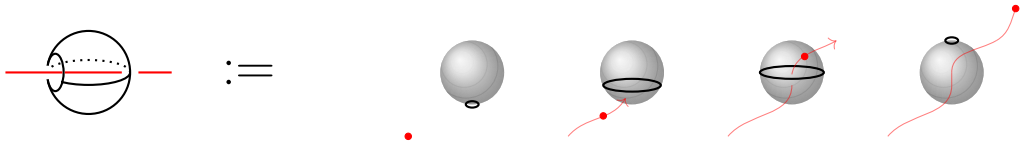

We can interpret ``addition on the 4D abacus'' by iteratively threading embedded spheres on a single thread, or in other words, connecting along the threads, as shown in Figure~\ref{fig:4dabacuseq}. There are two ways to obtain the number 2 from the number 1: by addition -- which is represented by iterative threading on the ``abacus'' thread as above, or by doubling, as explained below.

 Assuming the sphere is equipped with a normal vector field (a.k.a. framing, and we will define such a framing later), it makes sense to double the sphere along its framing. This operation will be denoted  by \begin{tikzpicture}[scale=.35,baseline=-.25cm]\inlinedeltasphere\end{tikzpicture} \threadedspheretube .
 For example, given the outward-pointing normal vector field, doubling the sphere results in two concentric spheres. In $\R^4$ two concentric spheres
 can be separated without intersecting each other. E.g, assume the coordinates are called $x$, $y$, $z$ and $t$, and two concentric spheres lie in the hyperplane $\{z=0\}$. Then one can continuously move the inner sphere into the hyperplane $z=1$, followed by moving it to a disjoint $t$-position from the outer sphere and then back to the $\{z=0\}$ hyperplane.

 Combining this with threading, we see that doubling a threaded sphere is the same as  the connected sum of two threaded spheres, as shown in Figure \ref{fig:4dabacuseq}. To simplify notation, we will denote the threaded sphere by $\mathbbm{1}$, and write $\mathbbm{1} \# \mathbbm{1}= \begin{tikzpicture}[scale=.35,baseline=-.15cm]\inlinedeltasphere\end{tikzpicture}  \mathbbm{1}  $.

 \begin{figure}
\begin{tikzpicture}[scale=.55]
\draw[thick] (0,0) circle (1cm);
\begin{scope}
    \clip (-1,0) rectangle (1,1);
\draw[ thick,dotted] (0,0) ellipse (1cm and .3cm);
\end{scope}
\begin{scope}
    \clip (-.64,0) rectangle (1,-1);
    \draw[thick] (0,0) ellipse (1cm and .3cm);
\end{scope}
\begin{scope}
    \clip (-1,1) rectangle (1,-1);
    \draw[thick] (-.8,0) ellipse (.2cm and .45cm);
\end{scope}

\draw [white,fill=white]  (-1.05,.15) rectangle (-.97,-.15);
\draw[red, thick ] (-2,0)--(-.63,0)(-.57,0)--(.7,0)(1.2,0)--(2,0);

\node[] at (2.5,0){\large $\#$};

\begin{scope}[xshift=5cm]
\draw[thick] (0,0) circle (1cm);
\begin{scope}
    \clip (-1,0) rectangle (1,1);
\draw[ thick,dotted] (0,0) ellipse (1cm and .3cm);
\end{scope}
\begin{scope}
    \clip (-.64,0) rectangle (1,-1);
    \draw[thick] (0,0) ellipse (1cm and .3cm);
\end{scope}
\begin{scope}
    \clip (-1,1) rectangle (1,-1);
    \draw[thick] (-.8,0) ellipse (.2cm and .45cm);
\end{scope}

\draw [white,fill=white]  (-1.05,.15) rectangle (-.97,-.15);
\draw[red, thick ] (-2,0)--(-.63,0)(-.57,0)--(.7,0)(1.2,0)--(2,0);
\end{scope}

\node[] at (8.5,0){\huge $=$};

\begin{scope}[xshift=11.5cm]
\draw[thick] (0,0) circle (1cm);
\begin{scope}
    \clip (-1,0) rectangle (1,1);
\draw[ thick,dotted] (0,0) ellipse (1cm and .3cm);
\end{scope}
\begin{scope}
    \clip (-.64,0) rectangle (1,-1);
    \draw[thick] (0,0) ellipse (1cm and .3cm);
\end{scope}
\begin{scope}
    \clip (-1,1) rectangle (1,-1);
    \draw[thick] (-.8,0) ellipse (.2cm and .45cm);
\end{scope}

\draw [white,fill=white]  (-1.05,.15) rectangle (-.97,-.15);
\draw[red, thick ] (-2,0)--(-.63,0)(-.57,0)--(.7,0)(1.2,0)--(2,0);

\begin{scope}[xshift=3cm]
\draw[thick] (0,0) circle (1cm);
\begin{scope}
    \clip (-1,0) rectangle (1,1);
\draw[ thick,dotted] (0,0) ellipse (1cm and .3cm);
\end{scope}
\begin{scope}
    \clip (-.64,0) rectangle (1,-1);
    \draw[thick] (0,0) ellipse (1cm and .3cm);
\end{scope}
\begin{scope}
    \clip (-1,1) rectangle (1,-1);
    \draw[thick] (-.8,0) ellipse (.2cm and .45cm);
\end{scope}

\draw [white,fill=white]  (-1.05,.15) rectangle (-.97,-.15);
\draw[red, thick ] (-1.3,0)--(-.63,0)(-.57,0)--(.7,0)(1.2,0)--(2,0);
\end{scope}
\end{scope}


\begin{scope}[xshift=23cm]

\node[] at (-5,0) {\huge $=$};

 \begin{scope}[xshift=-3.5cm, scale=.65]
\deltasphere

\end{scope}

\draw[thick] (0,0) circle (1cm);
\begin{scope}
    \clip (-1,0) rectangle (1,1);
\draw[ thick,dotted] (0,0) ellipse (1cm and .3cm);
\end{scope}
\begin{scope}
    \clip (-.64,0) rectangle (1,-1);
    \draw[thick] (0,0) ellipse (1cm and .3cm);
\end{scope}
\begin{scope}
    \clip (-1,1) rectangle (1,-1);
    \draw[thick] (-.8,0) ellipse (.2cm and .45cm);
\end{scope}

\draw [white,fill=white]  (-1.05,.15) rectangle (-.97,-.15);
\draw[red, thick ] (-2,0)--(-.63,0)(-.57,0)--(.7,0)(1.2,0)--(2,0);

\end{scope}
\end{tikzpicture}
\captionof{figure}{``$1+1=2$'' on the 4D abacus.}\label{fig:4dabacuseq}
\end{figure}
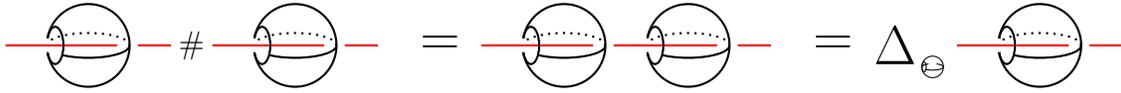

\subsection{w-Foams}\label{sec:Foams} In order to introduce the main ingredient $Z$, the homomorphic expansion, we need to place the threaded sphere in the more complex space of w-foams. We will briefly describe this space here and for more detail refer to \cite[Section 2]{WKO3} or \cite[Section 4.1]{WKO2}.

The space of w-foams, denoted $\wTFe$, is a {\em circuit algebra}, as defined in \cite[Section 2.4]{WKO2}. In short, circuit algebras are similar to the planar algebras of Jones \cite{Jones} but without the planarity requirement for the connection diagrams. In other words, a circuit algebra operation is to take some number of w-foams and connect their ends in an arbitrary (not necessarily planar) way, but respecting colour and orientation. For an example of a circuit algebra connection diagram, see Figure~\ref{fig:CAConn}. Circuit algebras are also close relatives of modular operads \cite{DHR}.

Each generator and relation of $\wTFe$ has a local topological interpretation in terms of certain {\em ribbon knotted tubes with foam vertices and strings} in $\R^4$.  Note that one dimensional strands cannot be knotted in $\R^4$, however, they can be knotted {\em with} two-dimensional tubes. In the diagrams, two-dimensional tubes will be denoted by \textbf{thick lines} and one dimensional strings by {\color{red}thin red lines}.

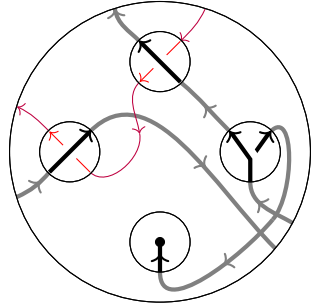
\begin{figure}
\begin{tikzpicture}

\draw[ultra thick, gray](.92,.28)to [in=320,out=135](.28,.92);
\draw[xshift=.6cm, yshift=.6cm,thick, gray,->] (0,0)--(-.01,.01);
\draw[ultra thick, gray](-.27,1.47)to [in=280,out=135](-.6, 1.9);
\draw[thick, gray,->](-.27,1.47)to [in=280,out=135](-.6, 1.9);
\draw[ultra thick, gray](-.93,.27)to [in=135,out=45](1.53,-1.283);
\draw[xshift=.6cm, yshift=-.2cm,thick, gray,->] (0,0)--(.01,-.01);
\draw[ultra thick, gray](0,-1.6)to [in=230,out=270](1.5,-.9);
\draw[ultra thick, gray](1.5,-.9)to [in=45,out=50](1.47,.27);
\draw[xshift=.85cm, yshift=-1.5cm,thick, gray,->] (0,0)--(-.01,-.007);
\draw[ultra thick, gray](-1.47,-.27)to [in=25,out=230](-1.9,-.6);
\draw[xshift=-1.6cm, yshift=-.4cm,thick, gray,->] (0,0)--(.01,.01);
\draw[ultra thick, gray](1.2,-.4)to [in=150,out=270](1.77,-.95);
\draw[xshift=1.3cm, yshift=-.66cm,thick, gray,->] (0,0)--(-.01,.008);

\draw[purple,->] (.6, 1.9)to [in=45,out=245] (.27,1.47);
\draw[purple,->] (-.27,.93)to [in=100,out=225] (-.27, .24);
\draw[purple,] (-.27, .24)to [in=315,out=280] (-.93,-.27);
\draw[purple,->] (-1.47,.27)to [in=335,out=135] (-1.9,.6);

\draw[] (0,0) circle (2cm);
\draw[fill=white](0, 1.2)circle (.4cm);
\draw[fill=white](0,-1.2)circle (.4cm);
\draw[fill=white](-1.2,0)circle (.4cm);
\draw[fill=white](1.2,0)circle (.4cm);


\draw[red, ->] (-.93,-.27)--(-1.47,.27);
\draw[thick, ->] (-1.47,-.27)--(-.91,.29);
\draw[line width=2.5mm,white] (-1.4,-.2)--(-1,.2);
\draw[ultra thick, ](-1.47,-.27)--(-.94,.26);


\draw[red, ->] (.27,1.47)--(-.27,.93);
\draw[line width=2.5mm,white] (.2,1)--(-.2,1.4);
\draw[thick, ->] (.27,.93)--(-.28,1.48);
\draw[ultra thick, ](.27,.93)--(-.25,1.45);

\draw[ultra thick, ] (1.2,-.1)-- (1.45,.25) ;
\draw[thick, ->] (1.2,-.1)-- (1.48,.28) ;

\draw[line width=2.5mm,white] (1.2,-.2)--(1.2,-.1)--(.95,.25) ;
\draw[ultra thick, ](1.2,-.4)--(1.2,-.1)--(.95,.25) ;
\draw[ thick,-> ](1.2,-.1)--(.92,.28) ;

\draw[fill= black ](0,-1.2)circle (.06cm);
\draw[ultra thick, ] (0,-1.6)--(0,-1.2);
\draw[thick,->] (0,-1.6)--(0,-1.35);

\draw[] (0,0) circle (2cm);
\draw[](0, 1.2)circle (.4cm);
\draw[](0,-1.2)circle (.4cm);
\draw[](-1.2,0)circle (.4cm);
\draw[](1.2,0)circle (.4cm);

\end{tikzpicture}
\caption{A circuit algebra connection diagram applied to two crossings, a vertex and a cap produces a larger w-foam when the inner circles are deleted.}\label{fig:CAConn}
\end{figure}

With this in mind, we define $\wTFe$ as a circuit algebra given in terms of generators and relations, and with some extra operations beyond circuit algebra composition. The generators, relations and operations are explained in detail in Sections~\ref{subsec:wgens} and~\ref{subsec:wrels}. The local topological interpretation of the generators and relations provides much of the intuition for this paper.

\[
  \wTFe=CA\!\left.\left.\left\langle
  \underbrace{\raisebox{-2mm}{$\stackrel{\rcross}{1}, \stackrel{\lcross}{2}, \stackrel{\capoff}{3}, \stackrel{\bdlambda}{4}, \stackrel{\bulambda}{5}, \stackrel{\mrcross}{6}, \stackrel{\mlcross}{7}, \stackrel{\rlambda}{8}, \stackrel{\mlambda}{9}$}}_{\text{generators}}
  \right|
  \underbrace{\parbox{1in}{\centering  $R1^s$ , $R2$, $R3$, $R4$, $OC$, $CP$ }}_{\text{relations}}
  \right|
  \underbrace{\raisebox{-2mm}{\centering $ u_e$}}_{\text{extra operation}}
  \right\rangle
\]



In \cite{WKO3} $\wTFe$ appears in its larger unoriented version (includes a {\em wen} and relations describing its behaviour) and it is equipped with more auxiliary operations (eg punctures, orientation switches). The expansion $Z$ constructed there is \textit{homomorphic} with respect to all of the operations in the appropriate sense. Here we focus only on orientable surfaces and the operations strictly needed for the Duflo problem -- the restriction of the $Z$ of \cite{WKO3} is a homomorphic expansion for this structure. In the following sections we will provide brief descriptions of $\wTFe$, its associated graded space of arrow diagrams, and the homomorphic expansion, to make this paper more self-contained.

\subsubsection{The generators of $\wTFe$}\label{subsec:wgens}
We begin by discussing the local topological meaning of each generator shown above. For more detail, see \cite[Sections 4.1.1 and 4.5]{WKO2}

 Knotted (more precisely, braided) tubes in $\R^4$ can equivalently be thought of as movies of flying circles in $\R^3$. The two crossings -- generators 1 and 2 -- stand for movies where two circles trade places as the circle corresponding to the under strand flies through the circle corresponding to the over strand entering from below. Note that our ``time'' flows from bottom to top. The bulleted end in generator 3 represents a tube ``capped off'' by a disk, or alternatively the movie where a circle shrinks to a point and disappears.

 Generators 4 and 5 stand for singular ``foam vertices'', and will be referred to as the positive and negative vertex, respectively. The positive vertex represents the movie shown in Figure~\ref{TrivVertex}: the right circle approaches the left circle from below, flies inside it and merges with it. The negative vertex represents a circle splitting and the inner circle flying out below and to the right.


The thin red strands denote one dimensional strings in $\R^4$, or ``flying points in $\R^3$''. The crossings between the two types of strands (generators 6 and 7) represent ``points flying through circles''. For example, generator 6, \mlcross,  \hspace{.5mm} stands for ``the point on the right approaches the circle on the left from below, flies through the circle and out to the left above it''. This explains why there are no generators with a thick strand crossing under a thin red strand: a circle cannot fly through a point.

 \begin{figure}
 \begin{tikzpicture}[scale=.7]
\draw[thick,blue] (1,-1) ellipse (.5cm and .1cm);
\draw[thick,blue] (.25,-.57) ellipse (.4cm and .066cm);

\draw[thick,blue] (0,-.1) ellipse (.25cm and .05cm);

\draw[ thick,purple] (0,-.1) ellipse (.5cm and .1cm);
\draw[thick,purple] (-1,-1) ellipse (.5cm and .1cm);
\draw[line width= .8 mm, white] (-.25,-.5) ellipse (.5cm and .1cm);
\draw[thick,purple] (-.25,-.5) ellipse (.5cm and .1cm);

\draw[thick,blue] (0,.3) ellipse (.4cm and .05cm);
\draw[ thick,purple] (0,.3) ellipse (.5cm and .1cm);

\draw[line width= .375 mm,  violet] (0,.65) ellipse (.5cm and .1cm);

\draw[line width= .375 mm,  violet] (0,1) ellipse (.5cm and .1cm);

\begin{scope}[xshift=4.25cm]
\draw[thick] (0,1) ellipse (.5cm and .1cm);
\draw[thick, dashed] (0,0) ellipse (.5cm and .1cm);
\begin{scope}
    \clip (-1,0) rectangle (1,-.6);
    \draw[thick] (0,0) ellipse (.5cm and .1cm);
\end{scope}
\draw[thick, dashed] (1,-1) ellipse (.5cm and .1cm);
\begin{scope}
    \clip (0,-1) rectangle (2,-1.6);
    \draw[thick] (1,-1) ellipse (.5cm and .1cm);
\end{scope}
\draw[thick, dashed] (-1,-1) ellipse (.5cm and .1cm);
\begin{scope}
    \clip (-2,-1) rectangle (0,-1.6);
    \draw[thick] (-1,-1) ellipse (.5cm and .1cm);
\end{scope}
\draw[thick] (-.5,1)--(-.5,0)--(-1.5,-1);
\draw[thick] (.5,1)--(.5,0)--(1.5,-1);
\draw[thick] (-.5,-1)--(.5,0);

\draw[thick] (-.1,-.4)--(-.5,0);
\draw[thick] (.1,-.6)--(.5,-1);

\end{scope}

\begin{scope}[xshift=8cm]
\draw[ line width= .6 mm, ->](-1,-1)--(0,0)--(0,1);
\draw[ line width= .6 mm, ](1,-1)--(.2,-.2);

\end{scope}

\node[] at (2,0){\Huge =};
\node[] at (6,0){\Huge =};

\end{tikzpicture}\hspace{1.5cm} \begin{tikzpicture}[scale=.75, yshift=1cm]

  \draw[thick] (-1,1) ellipse (.5cm and .1cm);
\draw[thick] (-1.5,1)to [in=140, out=270] (-.29,-.5);
\draw[thick] (-.5,1)to [in=130, out=270] (.7,-.49);
\draw[thick, dashed] (.2,-.5) ellipse (.5cm and .1cm);
\begin{scope}
    \clip (-1,-.5) rectangle (1,-1);
    \draw[thick] (.2,-.5) ellipse (.5cm and .1cm);
\end{scope}

\draw[red] (-1.5,-.7)to [in=230, out=30] (-.6,.2);
\draw[red, fill=red] (-.6,.2) circle (.05cm);

  \begin{scope}[xshift=3cm,yshift=.2cm]

\draw[line width= .6 mm,->](1, -1)--(0,0)--(0,1);
\draw[red](-.03,-.03)--(-1,-1);
  \end{scope}

\node[] at (1.5,0){\Huge =};
  \end{tikzpicture}
  \captionof{figure}{The trivalent vertices of $\wTFe$.} \label{TrivVertex}
  \end{figure}
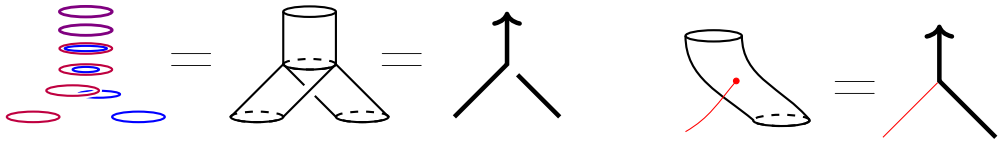

Generator 8 is a trivalent vertex of 1-dimensional strings in $\R^4$. Finally,  generator 9 is a ``mixed vertex'', in other words a one-dimensional string attached to the wall of a 2-dimensional tube. This is shown in Figure~\ref{TrivVertex}.

An important notion for later use is the {\em skeleton} of a w-foam. We give an intuitive definition here that is is sufficient for this paper; for a formal definition see \cite[Section 2.4]{WKO2}.
In general, viewing knotted objects as embeddings of circles, manifolds, graphs, etc, the skeleton is the embedded object without its embedding. In other words, the skeleton of a knotted object is obtained by allowing arbitrary crossing changes, or equivalently by replacing all crossings with ``virtual'' (or circuit) crossings. For example, the skeleton of an ordinary knot is a circle. The skeleton of the threaded sphere described above is the union of a sphere and an interval.

\subsubsection{The relations for $\wTFe$}\label{subsec:wrels}

 This section is a quick overview of the relations for $\wTFe$, which are described in detail in \cite[Section 4.5]{WKO2}.
The list of relations for $\wTFe$ is $\{$R$1^s$, R2, R3, R4, OC, CP$\}$;  Figure \ref{fig:R1OCCP} shows R$1^s$ and OC, and explains CP.
All relations have local 4-dimensional topological meaning, it is an instructive exercise to verify them. R$1^s$ stands for the weak (framed) version of the Reidemeister 1 move; R2 and R3 are the usual Reidemeister moves; and R4 allows moving a strand over or under a vertex.  OC stands for {\em Over-corssings Commute}, and CP for {\em Cap Pullout}. All relations should be interpreted in all sensible combinations of strand types: tube or string, and all orientations.

\begin{figure}
\begin{tikzpicture}[thick,scale=.65]
\draw (0,0) circle (.5cm);

\draw(0,1.5) circle (.5cm);

\draw[line width= 2.5 mm, white ](.5,-1) to [in=312, out=85](.357,.357);
\draw[->](.5,-1) to [in=310, out=85](.352,.352);

\draw(0,1.5) circle (.5cm);
\draw[white, fill=white] (.395,1.75) rectangle (1.25,1);
\draw[ yshift=1.5cm,-> ](.65,-1.5) to [in=310, out=65](.352,.352);

\draw[line width= 2 mm, white] (.357, 1.15) to [in=312, out=45] (.5,2.5);
\draw[->] (.357, 1.15) to [in=312, out=45] (.5,2.5);

\draw[->](2.5,-1)--(2.5,2.5);
\draw[<->, thin] (1.25,.75)--(2,.75);
\node[above] at (1.75,.75){\Small R$1^s$};

\end{tikzpicture}\hspace {1.5cm} \begin{tikzpicture}[thick,scale=.65]

\draw[thick] (-1.5,-1.5)--(1.5,1.5)(1.5,-1.5)--(-1.5,1.5);
\draw[line width= 2 mm,white] (0,-1.5) to [in=270, out=25] (1.25,0);
\draw[line width= 2 mm,white] (1.25,0) to  [in=335, out=90] (0,1.5);

\draw[thick] (0,-1.5) to [in=270, out=25] (1.25,0);
\draw[thick] (1.25,0) to  [in=335, out=90] (0,1.5);

\begin{scope}[xshift=5cm]
\draw[thick] (-1.5,-1.5)--(1.5,1.5)(1.5,-1.5)--(-1.5,1.5);
\draw[line width= 2 mm,white] (0,-1.5) to [in=270, out=155] (-1.25,0);
\draw[line width= 2 mm,white] (-1.25,0) to  [in=205, out=90] (0,1.5);

\draw[thick] (0,-1.5) to [in=270, out=155] (-1.25,0);
\draw[thick] (-1.25,0) to  [in=205, out=90] (0,1.5);

\end{scope}

\draw[<->, thin] (1.75,0)--(3.25,0);
\node[above] at (2.5,0){\Small OC};

\end{tikzpicture}


\begin{tikzpicture}[scale=.6]

\begin{scope}[thick]
\draw[thick](0,.75)--(0,-1.25);
\draw[fill=black] (0,.75) circle (.2cm);

\draw[line width= 2 mm,white](-1,0)--(1,0)(3,0)--(5,0);
\draw[thick](-1,0)--(1,0)(3.5,0)--(5.5,0);

\draw[thick](4.5,-.25)--(4.5,-1.25);
\draw[fill=black] (4.5,-.45) circle (.2cm);

\draw[<->, thin] (1.75,0)--(3.25,0);
\node[above] at (2.5,0){\Small CP};
\end{scope}

\begin{scope}[xshift=-.5cm]


\begin{scope}[xshift=1.5cm]
\begin{scope}[xshift=8cm]

\begin{scope}[yshift=1cm]

    \draw[] (0,0) ellipse (.75cm and .2cm);

\end{scope}
\begin{scope}[yshift=-2cm]
\begin{scope}
    \clip (-1,0) rectangle (1,1);
\draw[ dashed] (0,0) ellipse (.75cm and .2cm);
\end{scope}

\begin{scope}
    \clip (-1,0) rectangle (1,-1);
    \draw[] (0,0) ellipse (.75cm and .2cm);
\end{scope}

\end{scope}

\draw[](.75,-2)--(.75,1);


\end{scope}

\draw[] (8.45,-.75) ellipse (.1cm and .45cm);
\draw[dashed,thin] (8.45,-.75) ellipse (.04cm and .3cm);

\draw[white, fill=white, xshift=8.58cm, yshift=-.95cm] (0,0) rectangle (-.05,.1);
\draw[white, fill=white,xshift=8.58cm, yshift=-.7cm] (0,0) rectangle (-.05,.1);

\draw[white, fill=white,xshift=8.77cm, yshift=-.75cm] (0,0) rectangle (-.05,.1);
\draw[white, fill=white,xshift=8.77cm, yshift=-.97cm] (0,0) rectangle (-.05,.1);


\begin{scope}[xshift=10cm, yshift=-2cm]
\begin{scope}
    \clip (-1,0) rectangle (1,1);
\draw[ dashed] (0,0) ellipse (.5cm and .1cm);
\end{scope}

\begin{scope}
    \clip (-1,0) rectangle (1,-1);
    \draw[] (0,0) ellipse (.5cm and .1cm);
\end{scope}

\draw[line width= .7 mm, white] (-.5,0) to [in=350, out=110](-1.55,.9);
\draw[line width= .7 mm, white] (.5,0) to [in=350, out=110] (-1.55,1.6);
\draw[] (-.5,0) to [in=350, out=110](-1.6,.9);
\draw[] (.5,0) to [in=350, out=110] (-1.6,1.6);

\draw[dashed] (-2.4,1.6) ellipse (.07cm and .4cm);

\draw[dashed] (-1.6,1.6)to [in=-20, out=170](-2.4,2) ;
\draw[dashed] (-1.6,.9)to [in=-20, out=170](-2.4,1.2) ;

\draw[dashed] (-2.6,1.7) ellipse (.07cm and .4cm);

\draw[] (-2.6,2.1)to [in=290, out=170](-3.51,2.7) ;
\draw[] (-2.6,1.3)to [in=290, out=170](-4.49,2.7) ;

\end{scope}

\begin{scope}[xshift=7.75cm,scale=.7]

\begin{scope}[xshift=-2.5cm,yshift=1cm]
\begin{scope}
    \clip (-1,0) rectangle (1,2);
\draw[] (0,0) ellipse (.7cm and .45cm);
\end{scope}

\begin{scope}
    \clip (-1,0) rectangle (1,1);
\draw[ dashed] (0,0) ellipse (.7cm and .1cm);
\end{scope}

\begin{scope}
    \clip (-1,0) rectangle (1,-1);
    \draw[] (0,0) ellipse (.7cm and .1cm);
\end{scope}

\end{scope}

\end{scope}
\begin{scope}[xshift=8cm]
\draw[white,line width= .8 mm](-.75,-1.5)--(-.75,.75);
\draw[](-.75,-2)--(-.75,1);
\end{scope}



\begin{scope}[xshift=13cm]

\begin{scope}[yshift=1cm]

    \draw[] (0,0) ellipse (.75cm and .2cm);

\end{scope}
\begin{scope}[yshift=-2cm]
\begin{scope}
    \clip (-1,0) rectangle (1,1);
\draw[ dashed] (0,0) ellipse (.75cm and .2cm);
\end{scope}

\begin{scope}
    \clip (-1,0) rectangle (1,-1);
    \draw[] (0,0) ellipse (.75cm and .2cm);
\end{scope}

\end{scope}
\draw[](-.75,-2)--(-.75,1)(.75,-2)--(.75,1);

\begin{scope}[xshift=1.75cm,yshift=-1cm]
\begin{scope}[yshift=.25cm]
\begin{scope}
    \clip (-1,0) rectangle (1,2);
\draw[] (0,0) circle (.5cm);
\end{scope}

\begin{scope}
    \clip (-1,0) rectangle (1,1);
\draw[ dashed] (0,0) ellipse (.5cm and .1cm);
\end{scope}

\begin{scope}
    \clip (-1,0) rectangle (1,-1);
    \draw[] (0,0) ellipse (.5cm and .1cm);
\end{scope}
\end{scope}

\begin{scope}[yshift=-1cm]
\begin{scope}
    \clip (-1,0) rectangle (1,1);
\draw[ dashed] (0,0) ellipse (.5cm and .1cm);
\end{scope}

\begin{scope}
    \clip (-1,0) rectangle (1,-1);
    \draw[] (0,0) ellipse (.5cm and .1cm);
\end{scope}

\end{scope}
\draw[](-.5,-1)--(-.5,.25)(.5,-1)--(.5,.25);

\end{scope}

\draw[<->, thin] (-1.75,0)--(-3.25,0);
\node[above] at (-2.5,0){\Small CP};

\end{scope}

\end{scope} 

\end{scope}


\begin{scope}[xshift=1.5cm]
\begin{scope}[xshift=18cm, yshift=0cm, scale=1.25]

\draw[blue] (1.2,-1.6) ellipse (.4cm and .07cm); 
\draw[purple] (0,-1.6) ellipse (.5cm and .1cm);

\draw[blue] (.2,-1.1) ellipse (.4cm and .066cm);
\draw[line width= .6 mm, white] (0,-1) ellipse (.5cm and .1cm);
\draw[purple] (0,-1) ellipse (.5cm and .1cm);

\draw[blue] (0,-.5) ellipse (.35cm and .05cm);
\draw[line width= .6 mm, white] (0,-.5) ellipse (.5cm and .1cm);
\draw[purple] (0,-.5) ellipse (.5cm and .1cm);

\draw[ purple] (0,-.1) ellipse (.5cm and .1cm);
\draw[line width= .5 mm, white] (-.2,-.02) ellipse (.35cm and .06cm);
\draw[blue] (-.2,-.02) ellipse (.35cm and .06cm);

\draw[blue] (-.9,.3) ellipse (.25cm and .03cm);
\draw[ purple] (0,.3) ellipse (.5cm and .1cm);

\draw[blue] (-1,.65) ellipse (.1cm and .02cm);
\draw[ purple ] (0,.65) ellipse (.5cm and .1cm);

\end{scope} 

\begin{scope}[xshift=22cm, yshift=0cm, scale=1.25]
\draw[blue] (1.3,-1.5) ellipse (.5cm and .1cm); 
\draw[blue] (1.3,-1) ellipse (.35cm and .066cm);
\draw[blue] (1.3,-.55) ellipse (.25cm and .05cm);
\draw[blue] (1.3,-.1) ellipse (.08cm and .02cm);

\draw[purple] (0,-1.5) ellipse (.5cm and .1cm); 
\draw[purple] (0,-1) ellipse (.5cm and .1cm);
\draw[purple] (0,-.55) ellipse (.5cm and .1cm);
\draw[ purple] (0,-.1) ellipse (.5cm and .1cm);
\draw[ purple] (0,.3) ellipse (.5cm and .1cm);
\draw[ purple ] (0,.65) ellipse (.5cm and .1cm);

\end{scope}

\draw[<->, thin] (19.2,0)--(20.8,0);
\node[above] at (20,0){\Small CP};

\end{scope}
\end{tikzpicture}
\captionof{figure}{The relations $R1^s$ and $OC$ are shown.  $CP$ is explained with broken surface diagrams and as a movie of flying circles.}\label{fig:R1OCCP}
\end{figure}
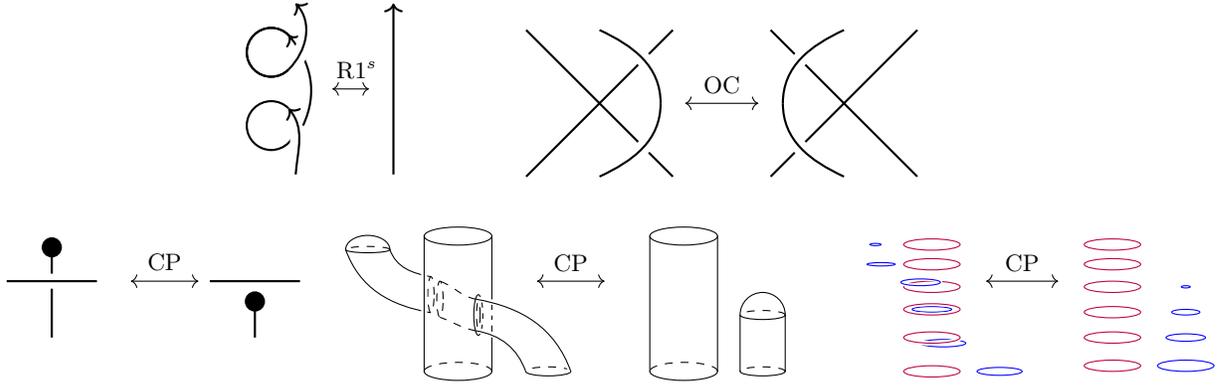

Note that all relations are circuit algebra relations. For example, the relation OC is understood as a relationship between two specific circuit diagram compositions of \lcross \hspace{2mm} and \rcross , as shown in Figure~\ref{CAOC}.

\begin{figure}
\begin{tikzpicture}
\draw[] (0,0) circle (1.5cm);
\draw[](.5, .5)circle (.35cm);
\draw[](.5,-.5)circle (.35cm);

\draw[ultra thick, gray](1.3,0.75)to [in=45,out=160](.75,0.75);
\draw[ultra thick, gray](0,1.5)to [in=135,out=270](.25,.75);
\draw[ultra thick, gray](-1.3,.75)to [in=135,out=315](.25,-.25);
\draw[ultra thick, gray](-1.3,-.75)to [in=225,out=45](.25,.25);
\draw[ultra thick, gray](0,-1.5)to [in=225,out=90](.25,-.75);
\draw[ultra thick, gray](1.3,-.75)to [in=315,out=200](.75,-0.75);
\draw[ultra thick, gray](.75,.25)to [in=45,out=315](.75,.-.25) ;

\draw[xshift=.1cm, yshift=.11cm,thick, gray,->] (0,0)--(.01,.01);
\draw[xshift=1.05cm, yshift=.82cm,thick, gray,->] (0,0)--(.01,-.003);
\draw[xshift=.03cm, yshift=1.1cm,thick, gray,->] (0,0)--(-.005,.01);
\draw[xshift=-1cm, yshift=.51cm,thick, gray,->] (0,0)--(-.01,.01);
\draw[xshift=-1cm, yshift=-.51cm,thick, gray,->] (0,0)--(.01,.01);
\draw[xshift=.07cm, yshift=-1cm,thick, gray,->] (0,0)--(.01,.02);
\draw[xshift=.85cm, yshift=.11cm,thick, gray,->] (0,0)--(.00,.01);
\draw[xshift=1cm, yshift=-.85cm,thick, gray,->] (0,0)--(-.01,-.005);

\begin{scope}[xshift=.5cm,yshift=.5cm,scale=.23]

\draw[thick, ->](1,-1)--(-1.1,1.1);
\draw[ ultra thick](-1,-1)--(-.2,-.2);
\draw[ultra thick] (.2,.2)--(1,1);
\draw[ thick,->] (.2,.2)--(1.1,1.1);

\draw[line width=2.5mm,white] (.5,-.5)--(-.5,.5);
\draw[ultra thick](1,-1)--(-1,1);

\end{scope}

\begin{scope}[xshift=.5cm,yshift=-.5cm, scale=.23]
\draw[ thick,->](-1,-1)--(1.1,1.1);
\draw[ultra thick](1,-1)--(.2,-.2);
\draw[ultra thick] (-.2,.2)--(-1,1);
\draw[ thick,->] (-.2,.2)--(-1.1,1.1);

\draw[line width=2.5mm,white] (-.5,-.5)--(.5,.5);
\draw[ ultra thick](-1,-1)--(1,1);

\end{scope}

\begin{scope}[xshift=4cm]
\draw[] (0,0) circle (1.5cm);
\draw[](-.5, .5)circle (.35cm);
\draw[](-.5,-.5)circle (.35cm);

\draw[ultra thick, gray](1.3,0.75)to[in=45,out=225](-.25,.-.25);
\draw[ultra thick, gray](0,1.5)to[in=45,out=270](-.25,0.75);
\draw[ultra thick, gray](-1.3,.75)to[in=135,out=30](-.75,.75);
\draw[ultra thick, gray](-1.3,-.75)to[in=225,out=-30](-.75,-.75);
\draw[ultra thick, gray](0,-1.5)to[in=315,out=90](-.25,-0.75);
\draw[ultra thick, gray](1.3,-.75)to[in=315,out=135](-.25,.25);
\draw[ultra thick, gray](-.75,.25)to[in=135,out=225](-.75,.-.25) ;

\draw[xshift=-.1cm, yshift=-.11cm,thick, gray,->] (0,0)--(.01,.01);
\draw[xshift=1cm, yshift=.5cm,thick, gray,->] (0,0)--(.01,.01);
\draw[xshift=1cm, yshift=-.5cm,thick, gray,->] (0,0)--(-.01,.01);
\draw[xshift=-.15cm, yshift=.15cm,thick, gray,->] (0,0)--(-.01,.01);
\draw[xshift=-.03cm, yshift=1.2cm,thick, gray,->] (0,0)--(.001,.01);
\draw[xshift=-.025cm, yshift=-1.2cm,thick, gray,->] (0.001,0)--(.001,.01);
\draw[xshift=-.85cm, yshift=0cm,thick, gray,->] (0.001,0)--(.001,.01);
\draw[xshift=-1cm, yshift=-.85cm,thick, gray,->] (0,0)--(.03,-.01);
\draw[xshift=-1cm, yshift=.85cm,thick, gray,->] (0,0)--(-.03,.01);

\begin{scope}[xshift=-.5cm,yshift=-.5cm,scale=.23]

\draw[thick, ->](1,-1)--(-1.1,1.1);
\draw[ ultra thick](-1,-1)--(-.2,-.2);
\draw[ultra thick] (.2,.2)--(1,1);
\draw[ thick,->] (.2,.2)--(1.1,1.1);

\draw[line width=2.5mm,white] (.5,-.5)--(-.5,.5);
\draw[ultra thick](1,-1)--(-1,1);
\end{scope}

\begin{scope}[xshift=-.5cm,yshift=.5cm, scale=.23]

\draw[ thick,->](-1,-1)--(1.1,1.1);
\draw[ultra thick](1,-1)--(.2,-.2);
\draw[ultra thick] (-.2,.2)--(-1,1);
\draw[ thick,->] (-.2,.2)--(-1.1,1.1);

\draw[line width=2.5mm,white] (-.5,-.5)--(.5,.5);
\draw[ ultra thick](-1,-1)--(1,1);
\end{scope}


\end{scope} 

\node[]at (2,0){\Huge =};
\node[above ]at (2,.2){\Small OC};

\end{tikzpicture}
\captionof{figure}{The OC relation written as a circuit algebra relation between two crossings.}\label{CAOC}
\end{figure}
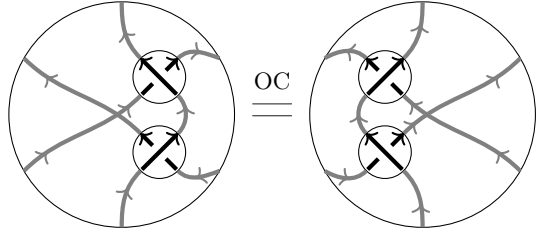

The circuit algebra $\wTFe$ is conjectured to be  a Reidemeister theory for \textit{ribbon} knotted tubes in $\R^4$ with caps, singular foam vertices and strings.  Here \textit{ribbon} means that the tubes have ``filling'' in $\mathbb{R}^4$ with only restricted types of singularities, for details see \cite[Section 2.2.2]{WKO1}. All the relations represent local topological statements: for example, Reidemeister 2 with a thin red bottom strand holds because the movie consisting of a point flying in through a circle and then immediately flying back out is isotopic to the movie in which the point and circle stay in place. However, it is an open question whether the known relations are sufficient. A similar Reidemeister theory has been proven for w-braids,which exhibits a simpler structure than $\wTFe$: \cite[Proposition 3.3]{BH} and \cite{Gol, Satoh}. For an explanation of the difficulties that arise for knots and tangles, see \cite[Introduction]{WKO2}.

\subsubsection{The operations on $\wTFe$ }\label{subsec:wops}
 In addition to the circuit algebra structure, $\wTFe$ is equipped with a set of auxiliary operations. Of these, in this paper we only use \textit{disc unzip}.

 The \textit{disc unzip operation} $u_e$ is defined for a capped strand labeled by $e$. Using the blackboard framing, $u_e$ doubles the capped strand $e$ and then attaches the ends of the doubled strand to the connecting ones, as shown Figure \ref{fig:unzip}.

 Topologically, the blackboard framing of the diagram induces a framing of the corresponding tubes and discs in $\R^4$ via Satoh's tubing map \cite[Section 3.1.1]{WKO1} and \cite{Satoh}. Briefly, each point of a (thick black) strand represents a circle in $\R^4$ via the tubing map, and the blackboard framing induces a "companion circle" (not linked with the original circle). A framed tube in $\R^4$ can be understood as a movie of flying circles in $\R^4$ with companions. Unzip is the operation ``pushing each circle off of itself slightly in the direction of the companion circles''. See also \cite[Section 4.1.3]{WKO2} for details on framings and unzips.

A related operation not strictly necessary for this paper, {\it strand unzip}, is defined for strands which end in two vertices of opposite signs, as shown in the right of Figure \ref{fig:unzip}.  For the interested reader a detailed definition of crossing and vertex signs is in \cite[Sections 3.4 and 4.1]{WKO2}. Strand unzip doubles the strand in the direction of the blackboard framing, and connects the ends of the doubled strands to the corresponding edge strands. Topologically, strand unzip pushes the tube off in the direction of the blackboard framing, as before.

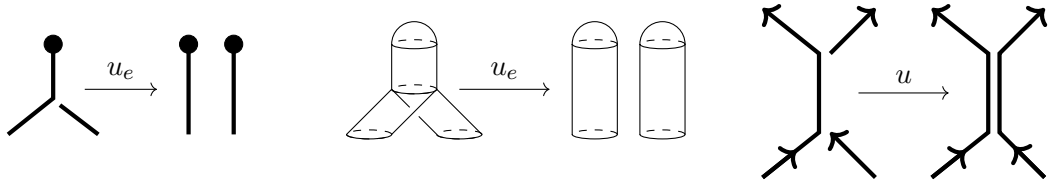
\begin{figure}
\begin{tikzpicture}[scale=.6]

\filldraw[fill=black, draw=black] (0,1) circle (2mm) ;
\draw[ultra thick](-1,-1)--(0,-.2)--(0,1) (.15,-.35)--(1,-1);

\draw[->] (.7,0)--node[above]{$u_e$}(2.3,0);
\begin{scope}[xshift=3cm]

\filldraw[fill=black, draw=black] (0,1) circle (2mm) ;
\draw[ultra thick](0,-1)--(0,1);

\filldraw[fill=black, draw=black] (1,1) circle (2mm) ;
\draw[ultra thick](1,-1)--(1,1);

\end{scope}

\begin{scope}[xshift=8cm]
\begin{scope}[yshift=1cm]
\begin{scope}
    \clip (-1,0) rectangle (1,2);
\draw[] (0,0) circle (.5cm);
\end{scope}

\begin{scope}
    \clip (-1,0) rectangle (1,1);
\draw[ dashed] (0,0) ellipse (.5cm and .1cm);
\end{scope}

\begin{scope}
    \clip (-1,0) rectangle (1,-1);
    \draw[] (0,0) ellipse (.5cm and .1cm);
\end{scope}

\end{scope}

\draw[ dashed] (0,0) ellipse (.5cm and .1cm);
\begin{scope}
    \clip (-1,0) rectangle (1,-.6);
    \draw[] (0,0) ellipse (.5cm and .1cm);
\end{scope}
\draw[ dashed] (1,-1) ellipse (.5cm and .1cm);
\begin{scope}
    \clip (0,-1) rectangle (2,-1.6);
    \draw[] (1,-1) ellipse (.5cm and .1cm);
\end{scope}
\draw[ dashed] (-1,-1) ellipse (.5cm and .1cm);
\begin{scope}
    \clip (-2,-1) rectangle (0,-1.6);
    \draw[] (-1,-1) ellipse (.5cm and .1cm);
\end{scope}
\draw[] (-.5,1)--(-.5,0)--(-1.5,-1);
\draw[] (.5,1)--(.5,0)--(1.5,-1);
\draw[] (-.5,-1)--(.5,0);

\draw[] (-.1,-.4)--(-.5,0);
\draw[] (.1,-.6)--(.5,-1);

\end{scope}

\draw[->] (9,0)--node[above]{$u_e$}(11,0);
\begin{scope}[xshift=12cm]
\begin{scope}[yshift=1cm]
\begin{scope}
    \clip (-1,0) rectangle (1,2);
\draw[] (0,0) circle (.5cm);
\end{scope}

\begin{scope}
    \clip (-1,0) rectangle (1,1);
\draw[ dashed] (0,0) ellipse (.5cm and .1cm);
\end{scope}

\begin{scope}
    \clip (-1,0) rectangle (1,-1);
    \draw[] (0,0) ellipse (.5cm and .1cm);
\end{scope}
\end{scope}
\begin{scope}[yshift=-1cm]
\begin{scope}
    \clip (-1,0) rectangle (1,1);
\draw[ dashed] (0,0) ellipse (.5cm and .1cm);
\end{scope}

\begin{scope}
    \clip (-1,0) rectangle (1,-1);
    \draw[] (0,0) ellipse (.5cm and .1cm);
\end{scope}

\end{scope}
\draw[](-.5,-1)--(-.5,1)(.5,-1)--(.5,1);

\end{scope}

\begin{scope}[xshift=13.5cm]
\begin{scope}[yshift=1cm]
\begin{scope}
    \clip (-1,0) rectangle (1,2);
\draw[] (0,0) circle (.5cm);
\end{scope}

\begin{scope}
    \clip (-1,0) rectangle (1,1);
\draw[ dashed] (0,0) ellipse (.5cm and .1cm);
\end{scope}

\begin{scope}
    \clip (-1,0) rectangle (1,-1);
    \draw[] (0,0) ellipse (.5cm and .1cm);
\end{scope}
\end{scope}
\begin{scope}[yshift=-1cm]
\begin{scope}
    \clip (-1,0) rectangle (1,1);
\draw[ dashed] (0,0) ellipse (.5cm and .1cm);
\end{scope}

\begin{scope}
    \clip (-1,0) rectangle (1,-1);
    \draw[] (0,0) ellipse (.5cm and .1cm);
\end{scope}

\end{scope}
\draw[](-.5,-1)--(-.5,1)(.5,-1)--(.5,1);

\end{scope}

\end{tikzpicture}\hspace{7mm}\begin{tikzpicture}[scale=.75]
\draw[ultra thick, ->] (-1,-1)--(0,-.2)--(0,1.2)--(-1,2);
\draw[ultra thick, ->] (1,-1)--(.2,-.2);
\draw[ultra thick, ->] (.2,1.2)--(1,2);
\draw[ultra thick, ->] (-1,-1)--(-.4,-.52);

\draw[->] (.7,.5)--node[above]{$u $}(2.3,.5);
\node[] at (-.5,.5) {$$};

\begin{scope}[xshift=3cm]
\draw[ultra thick, ->] (-1,-1)--(0,-.2)--(0,1.2)--(-1,2);
\draw[ultra thick, ->] (1,-1)--(.2,-.2)--(.2,1.2)--(1,2);
\draw[ultra thick, ->] (-1,-1)--(-.4,-.52);
\draw[ultra thick, ->] (1,-1)--(.45,-.45);
\end{scope}

\end{tikzpicture}

\captionof{figure}{Disc unzip on the left and middle, strand unzip on the right.}\label{fig:unzip}

\end{figure}

\subsection{Interpreting ``$1+1=2$'' in w-foams}\label{eqfoams}
The threaded sphere of Section~\ref{sec:abacus} can be described in $\wTFe$ by the diagram \threadedsphere , since a doubly capped tube is in fact a sphere.
Recall that the ``4D abacus'' interpretation of ``$1+1=2$''  is $\mathbbm{1} \# \mathbbm{1}=\Delta_{\sphere} \mathbbm{1}$, where $\mathbbm{1}$ is the threaded sphere, and $\Delta_{\sphere}$  is the doubling of the sphere along a framing.

The connected sum $\#$ operation for the threaded sphere is the circuit algebra composition shown in Figure \ref{fig:consumdubsphere}.
Doubling strands is realized in $\wTFe$ using the unzip operation. However, since the unzip operations in $\wTFe$ require an unzipped strand to end in either a vertex and a cap, or two vertices, we need to define a new {\em sphere unzip} operation, which doubles a twice-capped strand) along the blackboard framing, also shown in Figure~\ref{fig:consumdubsphere}. In Section~\ref{sec:arrows} we will need to show that the {\em homomorphic expansion} of $\wTFe$ also respects this operation.
To summarize, the topological statement ``$1+1=2$'' expressed in $\wTFe$ is shown in Figure~\ref{fig:topstate}.

\begin{figure}
\begin{tikzpicture}[scale=.75]
\draw[] (-5.75,1.75) rectangle (5.75,-2.25);

\node[] at (0,1) {\textbf{Topological Statement}};
\node[] at (4.5,0) {or};

\begin{scope}[xshift=-3.1cm,yshift=0cm]
\begin{scope}[scale=.4]

\draw[line width= .2 mm] (0,0) circle (1cm);
\begin{scope}
    \clip (-1,0) rectangle (1,1);
\draw[line width= .2 mm, dotted] (0,0) ellipse (1cm and .3cm);
\end{scope}
\begin{scope}
    \clip (-.64,0) rectangle (1,-1);
    \draw[line width= .2 mm] (0,0) ellipse (1cm and .3cm);
\end{scope}
\begin{scope}
    \clip (-1,1) rectangle (1,-1);
    \draw[line width= .2 mm] (-.8,0) ellipse (.2cm and .45cm);
\end{scope}

\draw [white,fill=white]  (-1.05,.15) rectangle (-.97,-.15);
\draw[red,  ] (-2,0)--(-.63,0)(-.57,0)--(.8,0)(1.2,0)--(2,0);

\node[] at (2.7,0){ $\#$};

\begin{scope}[xshift=5.5cm]
\draw[line width= .2 mm] (0,0) circle (1cm);
\begin{scope}
    \clip (-1,0) rectangle (1,1);
\draw[line width= .2 mm, dotted] (0,0) ellipse (1cm and .3cm);
\end{scope}
\begin{scope}
    \clip (-.64,0) rectangle (1,-1);
    \draw[line width= .2 mm] (0,0) ellipse (1cm and .3cm);
\end{scope}
\begin{scope}
    \clip (-1,1) rectangle (1,-1);
    \draw[line width= .2 mm] (-.8,0) ellipse (.2cm and .45cm);
\end{scope}

\draw [white,fill=white]  (-1.05,.15) rectangle (-.97,-.15);
\draw[red,  ] (-2,0)--(-.63,0)(-.57,0)--(.8,0)(1.2,0)--(2,0);

\end{scope}

\node[] at (9,0){ $=$};

\end{scope}

\end{scope}

\begin{scope}[xshift=-6.5cm,yshift=0cm]
\begin{scope}[xshift=9cm,yshift=0cm, scale=.4]
\begin{scope}[xshift=-3cm]\deltasphere\end{scope}

\draw[line width= .2 mm] (0,0) circle (1cm);
\begin{scope}
    \clip (-1,0) rectangle (1,1);
\draw[line width= .2 mm, dotted] (0,0) ellipse (1cm and .3cm);
\end{scope}
\begin{scope}
    \clip (-.64,0) rectangle (1,-1);
    \draw[line width= .2 mm] (0,0) ellipse (1cm and .3cm);
\end{scope}
\begin{scope}
    \clip (-1,1) rectangle (1,-1);
    \draw[line width= .2 mm] (-.8,0) ellipse (.2cm and .45cm);
\end{scope}

\draw [white,fill=white]  (-1.05,.15) rectangle (-.97,-.15);
\draw[red,  ] (-2,0)--(-.63,0)(-.57,0)--(.8,0)(1.2,0)--(2,0);
\end{scope}

\end{scope}

\begin{scope}[xshift=-2.05cm,yshift=-1.3cm]
\begin{scope}[scale=.35]

\begin{scope}[xshift=-1.5cm]
\draw[thick](0,-1)--(0,1);
\draw[red](-1,0)--(-.4,0);
\draw[red] (.4,0)--(1,0);
\filldraw[fill=black, draw=black] (0,1) circle (2mm) ;
\filldraw[fill=black, draw=black] (0,-1) circle (2mm) ;
\end{scope}
\node[] at (1,0){ $\#$};

\begin{scope}[xshift=3.5cm]
\draw[thick](0,-1)--(0,1);
\draw[red](-1,0)--(-.4,0);
\draw[red] (.4,0)--(1,0);
\filldraw[fill=black, draw=black] (0,1) circle (2mm) ;
\filldraw[fill=black, draw=black] (0,-1) circle (2mm) ;

\end{scope}

\node[] at (7.3,0){ $=$};
\begin{scope}[xshift=10cm]
\node at (0,0) {\huge $\Delta$};
\begin{scope}[xshift=1.3cm, yshift=-1.1cm, scale=.5]
\draw[thick](0,-1)--(0,1);
\filldraw[fill=black, draw=black] (0,1) circle (3mm) ;
\filldraw[fill=black, draw=black] (0,-1) circle (3mm) ;
\end{scope}

\end{scope}

\begin{scope}[xshift=13.0cm]

\draw[thick](0,-1)--(0,1);
\draw[red](-1,0)--(-.4,0);
\draw[red] (.4,0)--(1,0);
\filldraw[fill=black, draw=black] (0,1) circle (2mm) ;
\filldraw[fill=black, draw=black] (0,-1) circle (2mm) ;
\end{scope}

\end{scope}
\end{scope}

\end{tikzpicture}
 \captionof{figure}{The topological statement in w-foams: the connected sum of two threaded spheres along the threads is the same as the sphere unzip of a threaded sphere.}\label{fig:topstate}
 \end{figure}
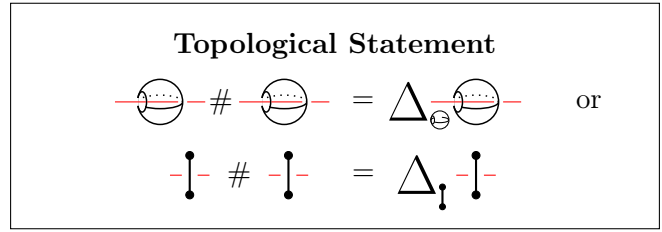

\section{Understanding the Diagramatic Statement}\label{sec:arrows}

 \subsection{The associated graded structure}
Let $\mathbb K$ be a field of characteristic zero. Extend $\wTFe$ by allowing formal linear combinations of w-foams of the same skeleton. As in \cite{WKO2, WKO3}, this algebraic structure -- whose operations consist of formal linear combinations, circuit algebra connections and unzips -- is filtered by powers of its augmentation ideal. Its associated graded structure, denoted $\calA$, is a space of {\em arrow diagrams on foam skeleta}. We will introduce arrow diagrams; for general background on the filtration and associated graded structure, see \cite[Section 2.2]{WKO2}; for the application to w-foams see \cite[Section 4.2]{WKO2}.

An {\em arrow diagram on a w-foam skeleton}, aka {\em Jacobi diagram}, consists of a w-foam skeleton (as defined in Section~\ref{subsec:wgens}), along with a uni-trivalent directed {\em arrow graph} with the following properties:
\begin{itemize}
\item univalent vertices are attached to the skeleton,
\item trivalent vertices are equipped with a cyclic orientation, and
\item each trivalent vertex is required to have two incoming arrows and one outgoing, this is referred to as the {\em two-in-one-out} rule.
\end{itemize}

By ``an arrow diagram'' we usually mean a formal $\mathbb K$-linear combination of arrow diagrams on the same skeleton.  An example is shown in Figure~\ref{fig:ArrowDiagram}; in figures the arrow graphs will be drawn in dotted lines. Arrow diagrams are combinatorial, not topological, objects, in other words it does not matter exactly how the arrow graph is drawn in the plane or where the univalent ends attach to the skeleton strands, only the order in which they are attached.

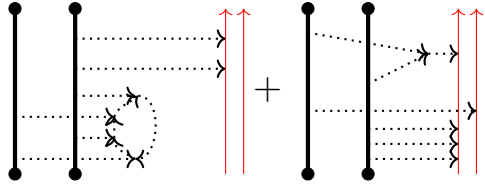
\begin{figure}
\begin{tikzpicture} [scale=.8]
\draw[ultra thick](-.5,-.25)--(-.5,2.5);
\draw[fill=black] (-.5,-.25) circle (.1cm);
\draw[fill=black] (-.5,2.5) circle (.1cm);
\draw[red,->](2,-.25)--(2,2.5);
\draw[red,->](2.3,-.25)--(2.3,2.5);

\begin{scope}[xshift=-1cm]
\draw[ultra thick](-.5,-.25)--(-.5,2.5);
\draw[fill=black] (-.5,-.25) circle (.1cm);
\draw[fill=black] (-.5,2.5) circle (.1cm);
\end{scope}

\draw[ ->,thick,dotted ] (-1.5,0)--(.5,0);
\draw[->,thick,dotted ] (-.5,.35)--(0.15,.35);
\draw[->,thick,dotted ] (-1.5,.7)--(.15,.7);
\draw[->,thick,dotted ] (-.5,1.05)--(.5,1.05);

\draw[->,thick,dotted ] (.5,0)to [out=150, in=270] (.15,.35);
\draw[->,thick,dotted ] (.15,.35)--(.15,.7);
\draw[->,thick,dotted ] (.15,.7) to [out=70, in=220] (.55,1.05);
\draw[->,thick,dotted ] (.55,1.05) to [out=0, in=0] (.5,0);

\draw[->,thick,dotted ] (-.5,1.5)--(2,1.5);
\draw[->,thick,dotted ] (-.5,2)--(2,2);
\node[]at (2.7,1.15){\Large $+$};
\end{tikzpicture}  \begin{tikzpicture} [scale=.8]
\draw[ultra thick](-.5,-.25)--(-.5,2.5);
\draw[fill=black] (-.5,-.25) circle (.1cm);
\draw[fill=black] (-.5,2.5) circle (.1cm);
\draw[red,->](2,-.25)--(2,2.5);
\draw[red,->](2.3,-.25)--(2.3,2.5);

\begin{scope}[xshift=1cm]
\draw[ultra thick](-.5,-.25)--(-.5,2.5);
\draw[fill=black] (-.5,-.25) circle (.1cm);
\draw[fill=black] (-.5,2.5) circle (.1cm);
\end{scope}

\draw[->,thick,dotted ] (-.5,.8)--(2.3,.8);

\draw[->,thick,dotted ] (.5,.5)--(2,.5);
\draw[->,thick,dotted ] (.5,.25)--(2,.25);
\draw[->,thick,dotted ] (.5,0)--(2,0);

\begin{scope}[yshift=.5cm]
\draw[->,thick, dotted ] (1.5,1.25)--(2,1.25);
\draw[->,thick,dotted ] (-.5,1.6)--(1.5,1.25);
\draw[->,thick,dotted ] (.5,.75)--(1.5,1.25);
\end{scope}

\end{tikzpicture}
\caption{A linear combination of arrow diagrams on a w-foam skeleton.}\label{fig:ArrowDiagram}
\end{figure}

The space $\calA$ consists of linear combinations of arrow diagrams as explained above, modulo a number of relations: $\aSTU$, $\aVI$, $RI$, $CP$ and $TF$. The $TF$ relation stands for Tails Forbidden on strings, and means just that: any diagram with an arrow tail ending on a red string is set to zero. The relations $RI$, $CP$ and $VI$ are shown in Figure~\ref{fig:RICPTF}. The acronym $RI$ stands for Rotation Invariance, this is the diagrammatic incarnation of the $R1^s$ relation; $CP$ stands for Cap Pull-out and comes from the identically named relation of w-foams; $VI$ stands for Vertex Invariance and has a less obvious topological explanation along the lines of the $4T$ relation of classical Vassiliev theory \cite{BN1}.

The $\aSTU$ relation is in fact a group of three relations, shown in Figure~\ref{fig:ASIHX}, and again analogous to the $STU$ relation of Vassiliev fame \cite{BN1}. The third $\aSTU$ relation is a special case also referred to as Tails Commute, or $TC$. The $\aSTU$ relations imply a number of other relations that are useful to be aware of: the $\aAS$ relation (or Anti-Symmetry of trivalent vertices), the $\aIHX$ relation, and the Four-Term or $\aft$ relation, all shown in Figure~\ref{fig:ASIHX}, and all of which mirror similar relations in the classical Vassiliev context. Note that the relations involving trivalent vertices are similar in spirit to the relations satisfied by a Lie bracket. This will become important and explicit in the next section.

\begin{figure}[h]
\begin{tikzpicture}[scale=.75]

\draw[ultra thick,->] (-1,-1)--(-1,1);
\draw[thick, dotted,](-1,-.5) to[out=0, in=270] (0,0);
\draw[thick,dotted,->](0,0) to[out=90, in=360](-1,.5);

\node[] at (.4,0.1) {$\stackrel{RI}{=}$};

\begin{scope}[xshift= 2cm]
\draw[ultra thick,->] (-1,-1)--(-1,1);
\draw[thick, dotted,->](0,0) to[out=270, in=3600] (-1,-.5);
\draw[thick, dotted,](-1,.5) to[out=0, in=90](0,0);
\end{scope}

\begin{scope}[xshift=5cm]
\draw[fill=black] (-1,.95) circle (.17cm);
\draw[ultra thick,] (-1,-1)--(-1,1);
\draw[thick,dotted,<-](-1,0) --(.4,0);

\node[] at (.95,0.1) {$\stackrel{CP}{=}$ 0};
\end{scope}

\end{tikzpicture}
\vspace{.5cm}

\begin{tikzpicture}[scale=.75]

\node at (-.75,0){$\pm$};
\draw[thick](-.75,-1)--(0,0)--(0,1)(.75,-1)--(0,0);
\draw[thick,->,dotted] (0,.9)to [out=225, in=0] (-.75,.75);

\node[xshift=1.4cm] at (-.75,0){$\pm$};
\draw[thick,xshift=2cm](-.75,-1)--(0,0)--(0,1)(.75,-1)--(0,0);
\draw[thick,->,dotted,xshift=2cm] (-.25,-.25)to [out=100, in=0] (-.75,.75);

\node[xshift=3cm] at (-.75,0){$\pm$};
\draw[thick,xshift=4cm](-.75,-1)--(0,0)--(0,1)(.75,-1)--(0,0);
\draw[thick,->,dotted,xshift=4cm] (.4,-.5)to [out=190, in=0] (-.75,.75);

\node[]at (5,0){$\stackrel{\aVI}{=} 0$, };
\node[]at (6.5,0){ and };

\node[xshift=6.5cm] at (-.75,0){$\pm$};
\draw[thick,xshift=9cm](-.75,-1)--(0,0)--(0,1)(.75,-1)--(0,0);
\draw[thick,->,dotted,xshift=9cm](-.75,.75) to [out=0, in=190] (0,.9);

\node[xshift=8cm] at (-.75,0){$\pm$};
\draw[thick,xshift=11cm](-.75,-1)--(0,0)--(0,1)(.75,-1)--(0,0);
\draw[thick,->,dotted,xshift=11cm](-.75,.75) to [out=0, in=110] (-.25,-.25);

\node[xshift=9.4cm] at (-.75,0){$\pm$};
\draw[thick,xshift=13cm](-.75,-1)--(0,0)--(0,1)(.75,-1)--(0,0);
\draw[thick,->,dotted,xshift=13cm](-.75,.75) to [out=0, in=200] (.4,-.5);

\node[]at (14,0){$\stackrel{\aVI}{=} 0$ };

\end{tikzpicture}
\captionof{figure}{The relations $RI$, $CP$,
 and $\protect \aVI$. Ambiguous strands can be either thick black or thin red. In the $\protect \aVI$ relation, signs are positive when the strand of the arrow ending is oriented towards the vertex, and negative otherwise.}\label{fig:RICPTF}
\end{figure}
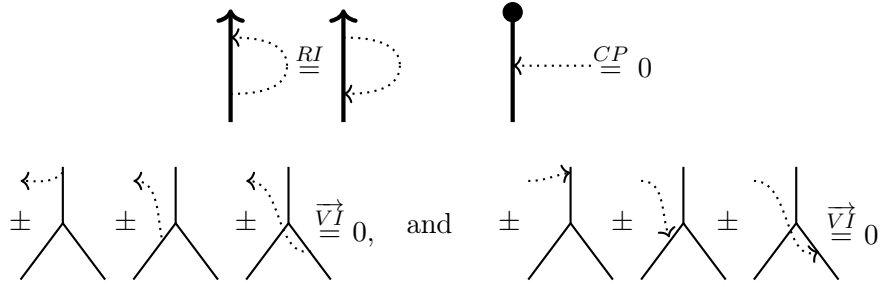

\begin{figure}[h]
\begin{tikzpicture}[scale=.75]

\draw[thick,->](-.75,0)--(.75,0);
\draw[thick,dotted,->](0,1)--(0,0);
\draw[thick,dotted,->](-.75,1.75)--(0,1);
\draw[thick,dotted,->](.75,1.75)--(0,1);

\node[] at (1.3,1){$\stackrel{{\aSTU_1}}{=}$};

\begin{scope}[xshift=2.5cm]
\draw[thick,->](-.75,0)--(.75,0);
\draw[thick,dotted,->](-.75,1.75)--(-.25,0);
\draw[thick,dotted,->](.75,1.75)--(.25,0);

\node[] at (1.4,1){$-$};

\end{scope}

\begin{scope}[xshift=5cm]
\draw[thick,->](-.75,0)--(.75,0);
\draw[thick,dotted,->](-.75,1.75)--(.25,0);
\draw[thick,dotted,->](.75,1.75)--(-.25,0);
\end{scope}

\begin{scope}[xshift=8cm]
\draw[ultra thick,->](-.75,0)--(.75,0);
\draw[thick,dotted,<-](0,1)--(0,0);
\draw[thick,dotted,->](-.75,1.75)--(0,1);
\draw[thick,dotted,<-](.75,1.75)--(0,1);

\node[] at (1.3,1){$\stackrel{{\aSTU_2}}{=}$};

\begin{scope}[xshift=2.5cm]
\draw[ultra thick,->](-.75,0)--(.75,0);
\draw[thick,dotted,->](-.75,1.75)--(-.25,0);
\draw[thick,dotted,<-](.75,1.75)--(.25,0);

\node[] at (1.4,1){$-$};

\end{scope}

\begin{scope}[xshift=5cm]
\draw[ultra thick,->](-.75,0)--(.75,0);
\draw[thick,dotted,->](-.75,1.75)--(.25,0);
\draw[thick,dotted,<-](.75,1.75)--(-.25,0);
\end{scope}

\end{scope}

\begin{scope}[xshift=16cm]
\draw[ultra thick,->](-.75,0)--(.75,0);
\draw[thick,dotted,<-](-.75,1.75)--(-.25,0);
\draw[thick,dotted,<-](.75,1.75)--(.25,0);

\node[] at (1.4,1){$\stackrel{{\aSTU_3}}{=}$};
\node[] at (1.35,0.5){\tiny $TC$};

\begin{scope}[xshift=2.5cm]
\draw[ultra thick,->](-.75,0)--(.75,0);
\draw[thick,dotted,<-](-.75,1.75)--(.25,0);
\draw[thick,dotted,<-](.75,1.75)--(-.25,0);
\end{scope}

\end{scope}

\end{tikzpicture}
\vspace{.75cm}

 \begin{tikzpicture}[scale=.75]
 \draw[thick,dotted,->](-1,1)--(-.02,.02);
 \draw[thick,dotted,->](1,1)--(0.02,0.02);
 \draw[thick,dotted,->](0,0)--(0,-1);

 \node[] at (1.5,0){$\stackrel{\aAS}{=} -$};
 \begin{scope}[xshift=3cm]
  \draw[thick,dotted,->](-1,1)to [out=360,in=15](.02,.02);
 \draw[thick,dotted,->](1,1)to [out=180,in=160] (-0.02,0.02);
 \draw[thick,dotted,->](0,0)--(0,-1);
 \end{scope}

 \begin{scope}[xshift=6.5cm]

  \draw[thick,dotted,->](-1,1)--(-.02,.02);
 \draw[thick,dotted,->](1,1)--(0.02,0.02);
 \draw[thick,dotted,->](0,0)--(0,-.5);
  \draw[thick,dotted,->](-1.01,-1.4)--(-.02,-.6);
   \draw[thick,dotted,->](.02,-.6)--(1.01,-1.4);
     \node[] at (1.25,0){$\stackrel{\aIHX}{=}$};

 \end{scope}
  \begin{scope}[xshift=9.5cm]

  \draw[thick,dotted,->](-1,1)--(-.51,-.23);
 \draw[thick,dotted,->](1,1)--(.52,-.23);
 \draw[thick,dotted,->](-.5,-.25)--(.5,-.25);
  \draw[thick,dotted,->](-1,-1.4)--(-.51,-.28);
   \draw[thick,dotted,->](.5,-.28)--(1.01,-1.4);
        \node[] at (1.65,-.2){$-$};

 \end{scope}

  \begin{scope}[xshift=12.5cm]

  \draw[thick,dotted,->](-1,1)--(.7,-.7);
 \draw[thick,dotted,->](1,1)--(-.7,-.7);

 \draw[thick,dotted,->](-.75,-.75)--(.7,-.75);

  \draw[thick,dotted,->](-1.25,-1.25)--(-.75,-.75);
   \draw[thick,dotted,->](.75,-.75)--(1.25,-1.25);
 \end{scope}

\end{tikzpicture}
\vspace{.5cm}

\begin{tikzpicture}[scale=.75]

\draw[ultra thick,->] (-1,-1)--(-1,1);
\draw[ultra thick,->] (0,-1)--(0,1);
\draw[thick,->] (1,-1)--(1,1);
\draw[dotted, thick,->](-1,.5)--(-.07,.5);
\draw[dotted,thick,->] (0,-.5)--(1,-.5);

\node[] at (1.75,0){+};

\begin{scope}[xshift=3.5cm]
\draw[ultra thick,->] (-1,-1)--(-1,1);
\draw[ultra thick,->] (0,-1)--(0,1);
\draw[thick,->] (1,-1)--(1,1);
\draw[dotted,thick,->](-1,.5)--(1.,.5);
\draw[dotted,thick,->] (0,-.5)--(1,-.5);
\node[] at (1.75,0){$\stackrel{\aft}{=}$};
\end{scope}

\begin{scope}[xshift=7cm]
\draw[ultra thick,->] (-1,-1)--(-1,1);
\draw[ultra thick,->] (0,-1)--(0,1);
\draw[thick,->] (1,-1)--(1,1);
\draw[dotted,thick,->](0,.5)--(1.,.5);
\draw[dotted,thick,->] (-1,-.5)--(0,-.5);
\node[] at (1.75,0){+};
\end{scope}

\begin{scope}[xshift=10.5cm]
\draw[ultra thick,->] (-1,-1)--(-1,1);
\draw[ultra thick,->] (0,-1)--(0,1);
\draw[thick,->] (1,-1)--(1,1);
\draw[dotted,thick,->](0,.5)--(1.,.5);
\draw[dotted,thick,->] (-1,-.5)--(1,-.5);
\end{scope}
\end{tikzpicture}

\captionof{figure}{The relations $\protect \aSTU$, $\protect \aAS$, $\protect \aIHX$ and $\protect \aft$ on $\mathcal{A}^{swt}$. Ambiguous strands can be either thick black or thin red.}\label{fig:ASIHX}
\end{figure}

We introduce the following notation: For a w-foam ${F}\in \wTFe$, let $\calA (S(F))$ denote the space of arrow diagrams with skeleton $S(F)$, where $S(F)$ is the skeleton of $F$ as defined in Section~\ref{sec:Foams}. Often we will write $\calA (F)$ to mean $\calA (S(F))$.

Arrow diagrams, just like w-foams, form a circuit algebra: given several arrow diagrams, a circuit algebra operation is to connect (some of) their skeleton ends by appropriately coloured and oriented lines. Indeed, circuit algebra composition in $\calA$ is the associated graded operation of the circuit algebra composition in $\wTFe$.

 As for the (strand, disc and sphere) unzip operations $u_e$, given a w-foam $F$ with a choice of a strand $e$, the associated graded unzip operation $u_e: \calA (F)\rightarrow \calA (u_e(F))$ maps each arrow ending on $e$ to a sum of two arrows, one ending on each of the two new strands which replace $e$. For example, an arrow diagram with $k$ arrows ending on $e$ -- either heads or tails -- is mapped to a sum of $2^k$ arrow diagrams. This sum is represented notationally as shown in Figure \ref{fig:exunzip}.

 \begin{figure}
 \[u_e \left( \hspace{2mm}
 \begin{tikzpicture}[baseline=0,scale=.75]

 \draw[ultra thick, ->]((-1,-1)--(-1,1);
 \draw[dotted, thick, <-](1,0)--(-1,0);
 \node[] at (-.75,-1){$e$};
 \end{tikzpicture}
\hspace{2mm} \right) =
\begin{tikzpicture} [baseline=0,scale=.75]
 \begin{scope}[xshift=4cm]
   \draw[ultra thick, ->]((-1,-1)--(-1,1);
     \draw[ultra thick, ->]((-1.25,-1)--(-1.25,1);
 \draw[dotted, thick, <-](1,0)--(-1,0);
  \node[] at (-.95,-1.25){\tiny{$e_1$ $e_2$}};
  \node[] at (2,0){+};
 \end{scope}

 \begin{scope}[xshift=8cm]
    \draw[ultra thick, ->]((-1.25,-1)--(-1.25,1);
 \draw[ultra thick, ->]((-1,-1)--(-1,1);
 \draw[dotted, thick, <-](1,0)--(-1.25,0);
     \node[] at (-.95,-1.25){\tiny{$e_1$ $e_2$}};
   \node[] at (2,0){$:=$};
   \node[] at (2,-.5) {\tiny{notation}};
 \end{scope}

  \begin{scope}[xshift=12.5cm]
    \draw[ultra thick, ->]((-1.25,-1)--(-1.25,1);
 \draw[ultra thick, ->]((-1,-1)--(-1,1);
 \draw[dotted, thick, <-](1,0)--(-.75,0);
 \draw[fill=white] (-1.5,.25) rectangle (-.75,-.25);
 \node[] at (-1.1,0){$+$};
 \end{scope}

 \end{tikzpicture}\]
 \captionof{figure}{ Unzipping the strand labeled $e$, where $e_1$ and $e_2$ are the two new strands replacing $e$. Here an arrow tail is shown, but the same applies to arrow heads.}\label{fig:exunzip}
 \end{figure}

 \subsection{The Homomorphic Expansion}\label{sec:Z} In this section we introduce the {\em homomorphic expansion} $Z: \wTFe \to \calA$, which exists by the results of
 \cite{WKO2} or \cite{WKO3}. An expansion is a filtered linear map\footnote{It is typically also required that $Z$ is {\em group-like}, which is the case for the homomorphic expansion of \cite{WKO2, WKO3}.} with the property that the associated graded map $\operatorname{gr} Z: \calA \to \calA$ is the identity map on $\calA$.  A homomorphic expansion is an expansion that is a circuit algebra homomorphism and also
intertwines each auxiliary operation (here only unzips) of $\wTFe$ with its arrow diagrammatic counterpart, meaning that the square below commutes. For disc unzip, this was shown in \cite{WKO2,WKO3}; for sphere unzip, we will prove it in Lemma~\ref{lem:UnzipHom}.
\[
\begin{tikzpicture}
\node at(0,0) {$\calA$};
\node at(3,0) {$\calA$};
\node at(0,2) {$\wTFe$};
\node at(3,2) {$\wTFe$};
\draw[->](.5,0)--node[below]{$u_e$} (2.5,0);
\draw[->](.5,2)--node[below]{$u_e$} (2.5,2);
\draw[->](0,1.5)--node[left]{$Z$} (0,.5);
\draw[->](3,1.5)--node[right]{$Z$} (3,.5);
\end{tikzpicture}
\]

The map $Z$  sends each generator $G$ to an infinite sum of arrow diagrams on the skeleton $S(G)$, that is, $Z(G)\in \calA(S(G))$. The values of the crossings and the cap are computed explicitly in \cite{WKO2, WKO3}; to refer to them we use the notation shown in Figure \ref{fig:ValuesOfZ}. In particular, the $Z$-value of a crossing of a black strand and a red string is the exponential $e^a$ of an arrow $a$, to be interpreted as the power series, where $a^n$ is shown in Figure~\ref{fig:ValuesOfZ}.

\begin{figure}
\[Z \left( \begin{tikzpicture}[scale=0.5,baseline=0mm]
\draw[ultra thick, ] (0,-1)--(0,1);
\draw[fill=black] (0,1) circle (.2cm);
\end{tikzpicture} \right)=\begin{tikzpicture} [scale=0.5,baseline=-1mm]
\draw[ultra thick, ] (0,-1)--(0,1);
\draw[fill=black] (0,1) circle (.2cm);
\draw[fill=white] (-.65,-.3) rectangle (.65,.3);
\node[] at (0,0) {\Small $C$};
\end{tikzpicture} \hspace{1.5cm} Z \left( \begin{tikzpicture}[scale=0.5,baseline=0mm]
\draw[ultra thick](1,-1)--(-1,1);
\draw[thick, ->](1,-1)--(-1.1,1.1);
\draw[red, thick](-1,-1)--(-.2,-.2);
\draw[ red, thick] (.2,.2)--(1,1);
\draw[red,  thick,->] (.2,.2)--(1.1,1.1);
\end{tikzpicture} \right)=\begin{tikzpicture} [scale=0.5,baseline=-1mm]
\draw[ultra thick,-> ] (0,-1.25)--(0,1.25);
\draw[->,thick,dotted] (0,0)--node[above] {$e^a$}(2.25,0);
\draw[thick,red,->] (2.25,-1.25)--(2.25,1.25);
\end{tikzpicture} \hspace{1.5cm}a^n:= \begin{tikzpicture} [scale=0.5,baseline=-1mm]
\draw[ultra thick,-> ] (0,-1.25)--(0,1.25);
\draw[thick,red,->] (2.25,-1.25)--(2.25,1.25);
\draw[->,thick,dotted] (0,-.5)--(2.25,-.5);
\draw[->,thick,dotted] (0,.5)--(2.25,.5);
\draw[->,thick,dotted] (0,-.7)--(2.25,-.7);
\node[] at(1.12,.2) {$\vdots$};
\node[] at(2.6,-.1) {\huge $\}$};
\node[] at(3.1,-.1) { \Small $n$};

\end{tikzpicture} \]\captionof{figure}{Values of $Z$ on generating w-foams.}\label{fig:ValuesOfZ}

\end{figure}

To describe the $Z$-value $C$ of a cap, we need to introduce a special class of arrow diagrams called {\em wheels}.  A wheel is an oriented cycle of arrows with a finite number of incoming arrows, or ``spokes''. (The 2-in-1-out property forces all univalent ends of this arrow graph to be arrow tails, that is, all spokes are incoming.) See Figure~\ref{WheelExample} for an example of a wheel. Note that reversing the orientation of an even wheel yields an equivalent arrow diagram via $\aAS$ and $TC$ relations, while reversing the orientation of an odd wheel produces its negative. Hence, from now on we assume all wheels are oriented clockwise. The value $C$ is an infinite sum of even wheels:
\begin{equation}\label{eq:C}
C= \operatorname{exp}\left(\sum_{n=1}^\infty c_{2n}w_{2n}\right), \quad \text{ where } \quad
\sum_{n=1}^\infty c_{2n}x^{2n}=\frac{1}{4} \operatorname{log}\frac{\operatorname{sinh}{x/2}}{x/2}.
\end{equation}
In particular, $c_2=\frac{1}{96}$, $c_4=-\frac{1}{11520}$, and $c_6=\frac{1}{752776}$.

Since $Z$ is a circuit algebra homomorphism, given the values of $Z$ on the generators, it is straightforward to compute $Z$ of any w-foam $F$: if $F$ is a circuit composition of some generators $\{G_i\}$,  then $Z(F)$ is the same circuit composition of the values $Z(G_i)$.

We are working towards a diagrammatic statement of ``$1+1=2$'', which depends heavily on the homomorphicity of $Z$. We have yet to prove that $Z$ is homomorphic with respect to the sphere unzip operation, and to achieve this we need to discuss the $Z$-values of the vertices in some detail.
By definition, $Z(\bdlambda)\in \calA(\bdlambda)$. Using iterative applications of the relation $\aVI$, all arrow endings on the vertical strand of \bdlambda \hspace{1mm} can be moved to the bottom two strands. This induces an isomorphism $ \calA(\bdlambda)\cong \calAtwo$ \cite{WKO2}. Thus, $Z(\bdlambda)$ can be viewed as an element of $\calAtwo$ denoted by $V_+$, as shown in Figure \ref{fig:vplus}. The arrow diagram $V_- =Z(\bulambda)\in \calAtwo$ is defined similarly.

Note  that $\calAtwo$ is an algebra\footnote{In fact,  a Hopf algebra with coproduct $\square:\calAtwo\rightarrow \calAtwo\otimes \calAtwo $. The coproduct $\square$ of an arrow diagram is a sum of all possible ways of attaching each of the connected component of the arrow graph -- after removing the skeleton -- to one of the tensor factor skeleta. Details on the Hopf algebra structure are in \cite[Section 3.2] {WKO2}.} with multiplication given by vertical concatenation. Let us recall a useful fact from~\cite{WKO2}:

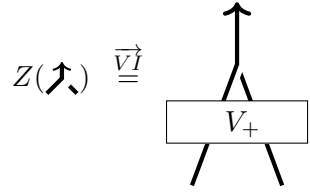
\begin{figure}
\[Z(\bdlambda) \hspace{2mm} \stackrel{\aVI}{=} \hspace{2mm}
 \begin{tikzpicture}[scale=0.75,baseline=-1.75mm]
 \draw[ultra thick ] (.8,-2)--(0,.15);
 \draw[line width= 2 mm, white ] (-.8,-2)--(0,.15)--(0,1.25);
\draw[ultra thick,-> ] (-.8,-2)--(0,.15)--(0,1.25);
\draw[fill=white,yshift=.5cm] (-1.25,-1) rectangle (1.25,-1.75);
\node[yshift=.25cm]at(0.1,-1.25) {$V_+$};
\end{tikzpicture}\]
\captionof{figure}{The definition of $V_+$.}\label{fig:vplus}
\end{figure}

\begin{lem}\label{lemV}
 In  $\calAtwo$, $V_+$ and $V_-$ are multiplicative inverses, i.e.  $V_+\cdot V_-=1=$ \begin{tikzpicture}[scale=.2, baseline=.0mm  ]
\draw[ultra thick,](0,-.5)--(0,1.2);
\draw[ultra thick,](1,-.5)--(1,1.2);
\draw[thick,->](0,-.5)--(0,1.3);
\draw[thick,->](1,-.5)--(1,1.3);
\end{tikzpicture}.
\end{lem}

\begin{proof}
This follows from the fact that $Z$ is homomorphic with respect to the unzip operation, as shown in Figure \ref{fig:Vproof}.
\end{proof}

\begin{figure}
\[u_eZ \left( \begin{tikzpicture}[scale=0.5,baseline=1.75mm]
\draw[ultra thick, ->] (-1,-1)--(0,-.2)--(0,1.2)--(-1,2);
\draw[ultra thick, ->] (1,-1)--(.2,-.2);
\draw[ultra thick, ->] (.2,1.2)--(1,2);
\draw[ultra thick, ->] (-1,-1)--(-.4,-.52);
\node[]at(-.5,.5){$e$};
\end{tikzpicture} \right)=u_e \left(\begin{tikzpicture}[scale=0.5,baseline=1.75mm]
\draw[ultra thick, ] (-1,-1)--(0,-.2)--(0,1.2)--(-1,2);
\draw[ultra thick, ] (1,-1)--(.2,-.2);
\draw[ultra thick, ] (.2,1.2)--(1,2);
\node[]at(-.5,.5){$e$};

\draw[fill=white] (-1.25,2) rectangle (1.25,2.75);
\draw[ultra thick,->](-1,2.75)--(-1,3.3);
\draw[ultra thick,->](1,2.75)--(1,3.3);
\node[]at(0.1,2.3) {\Small $V_-$};

\draw[fill=white] (-1.25,-1) rectangle (1.25,-1.75);
\draw[ultra thick,<-](-1,-1.75)--(-1,-2.3);
\draw[ultra thick,<-](1,-1.75)--(1,-2.3);
\node[]at(0.1,-1.4) {\Small $V_+$};

\end{tikzpicture}\right)= \begin{tikzpicture}[scale=0.5,baseline=1.75mm]
\draw[ultra thick,-> ] (-.8,-2)--(-.8,3);
\draw[ultra thick,-> ] (.8,-2)--(.8,3);

\draw[fill=white,yshift=-.5cm] (-1.25,2) rectangle (1.25,2.75);
\node[yshift=-.3cm]at(0.1,2.4) {\Small $V_-$};

\draw[fill=white,yshift=.5cm] (-1.25,-1) rectangle (1.25,-1.75);
\node[yshift=.25cm]at(0.1,-1.4) {\Small $V_+$};
\end{tikzpicture}\hspace{1.5cm} Zu_e \left( \begin{tikzpicture}[scale=0.5,baseline=1.75mm]
\draw[ultra thick, ->] (-1,-1)--(0,-.2)--(0,1.2)--(-1,2);
\draw[ultra thick, ->] (1,-1)--(.2,-.2);
\draw[ultra thick, ->] (.2,1.2)--(1,2);
\draw[ultra thick, ->] (-1,-1)--(-.4,-.52);
\node[]at(-.5,.5){$e$};
\end{tikzpicture} \right)= Z\left( \begin{tikzpicture}[scale=0.5,baseline=1.75mm]
\draw[ultra thick,-> ] (-.8,-1)--(-.8,2);
\draw[ultra thick,-> ] (.8,-1)--(.8,2);
\end{tikzpicture}\right)= \hspace{.1cm}\begin{tikzpicture}[scale=0.5,baseline=1.75mm]
\draw[ultra thick,-> ] (-.6,-1)--(-.6,2);
\draw[ultra thick,-> ] (.8,-1)--(.8,2);
\end{tikzpicture}\]
\captionof{figure}{Visual proof of Lemma \ref{lemV}.}\label{fig:Vproof}
\end{figure}
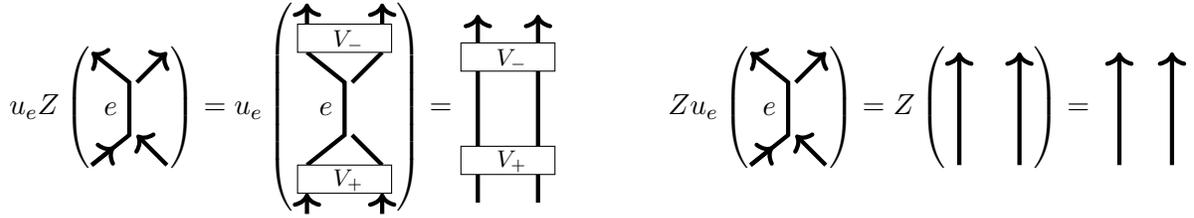

The following corollary is a crucial ingredient in proving that $Z$ is also homomorphic with respect to sphere unzip:
 \begin{cor}\label{cor:InverseVertices}
$Z(\begin{tikzpicture}[scale=.15,, baseline=-2.5mm]
\draw[ultra thick](-1.12,-1.12)--(0,0)--(0,1) (.35,-.35)--(1.1,-1.1);
\draw[ ultra thick](-1.05,-.95)--(0,-2)--(0,-3)(1.05,-.95)--(.25,-1.65);
\end{tikzpicture})=\begin{tikzpicture}[scale=.15,, baseline=-2.5mm]
\draw[ultra thick](-1.12,-1.12)--(0,0)--(0,1) (.35,-.35)--(1.1,-1.1);
\draw[ ultra thick](-1.05,-.95)--(0,-2)--(0,-3)(1.05,-.95)--(.25,-1.65);
\end{tikzpicture}$, that is, the $Z$-value of this w-foam is trivial, a skeleton with no arrows. \qed
\end{cor}

\begin{lem}\label{lem:UnzipHom}
Any homomorphic expansion $Z$ of $\wTFe$ is also homomorphic with respect to sphere unzip.
\end{lem}

\begin{proof}
To prove this we realize the sphere unzip operation as a composition of disc unzips with another new operation called ``cactus grafting''. Let $K$ denote
the ``cactus'' w-foam shown in Figure~\ref{fig:CactusGraft}. If $e$ is a capped strand of a w-foam, the {\em cactus grafting} operation $g_e$ attaches $K$ to the capped end of $e$. This is well-defined: the only relation involving a cap is $CP$, namely, the cap can be pulled out from under a strand. The same is true for the cactus, combining $CP$ with $R4$ moves. Topologically, cactus grafting is cutting a small disk out of a capped strand and gluing the resulting boundary circle to the boundary of the cactus $K$.

The associated graded cactus grafting operation on arrow diagrams -- also denoted $g_{e}$ -- simply attaches the skeleton $K$ (without any arrows) at the end of the capped strand $e$ in place of the cap. This is well-defined: the cap participates only in the $CP$ relation, that is, an arrow head that is not separated from the cap by another arrow ending is zero. However, an arrow head on $K$ that is not separated by another arrow ending from the bottom of $K$ is also zero, as shown in Figure~\ref{fig:CactusGraft}; this property is known as the {\em head invariance of arrow diagrams} \cite[Remark 3.14]{WKO2}.

Finally, Corollary~\ref{cor:InverseVertices} implies that $Z(K)$ is simply a cap value $C$ at the top of an otherwise empty skeleton $K$. Since the value $C$ can be move from the top to the bottom of $K$ using two $\aVI$ relations, we obtain that for any w-foam $F$ with a capped edge $e$, $Z(g_{e}(F))=g_{e}(Z(F))$, in other words $Z$ is homomorphic with respect to $g_{e}$.

Finally, sphere unzip can be written as a composition of one cactus grafting followed by two disc unzips as shown in Figure \ref{fig: cactuscomp}, hence $Z$ is homomorphic with respect to sphere unzip.
\end{proof}

\begin{figure}
\begin{tikzpicture}[scale=.4]

\begin{scope}[xshift=3cm]
\draw[ultra thick] (5,-1)--(6,-2)--(5,-3);
\draw[line width= 3 mm, white](5,1)--(5,-1)--(4,-2)--(5,-3)--(5,-4);
\draw[ultra thick] (5,1)--(5,-1)--(4,-2)--(5,-3)--(5,-4);
\filldraw[fill=black, draw=black] (5,1) circle (2mm) ;
\node[right] at (5, 0) {$K$};
\end{scope}

\begin{scope}[xshift=7cm]
\draw[ultra thick] (8,-4)--(8,-2);
\filldraw[fill=black, draw=black] (8,-2) circle (2mm) ;

\draw[ultra thick](13,0)--(14,-1)--(13,-2);
\draw[line width= 3 mm, white](13,1)--(13,0)--(12,-1)--(13,-2)--(13,-4);
\draw[ultra thick](13,1)--(13,0)--(12,-1)--(13,-2)--(13,-4);
\filldraw[fill=black, draw=black] (13,1) circle (2mm) ;

\draw[->] (9,-3)--node[above]{\Large $g$}(12,-3);
\node[] at (11,-2.5) {\small $e$};
\node[] at (7.5,-2.9) {\Large $e$};
\end{scope}

\begin{scope}[xshift=15cm]
\draw[ultra thick](13,0)--(14,-1.25)--(13,-2.5);
\draw[line width= 3 mm, white](13,1)--(13,0)--(12,-1.25)--(13,-2.5)--(13,-4);
\draw[ultra thick](13,1)--(13,0)--(12,-1.25)--(13,-2.5)--(13,-4);
\filldraw[fill=black, draw=black] (13,1) circle (2mm) ;

\draw[blue,thick] (11.5,-2) rectangle (14.5,-.5);
\node[blue] at (13,-1.12) {$D$};

\draw[thick,dotted ,->](15.5,-3.25)--(13.1,-3.25);

\node[] at(17,-1.12){\Large $=0$};

\end{scope}

\end{tikzpicture}
\captionof{figure}{The ``cactus'' w-foam $K$ on the left, cactus grafting in the middle. The ``head invariance'' property of arrow diagrams on $K$ is shown on the right, where $D$ denotes any arrow diagram above the arrow head on the skeleton $K$.}\label{fig:CactusGraft}
\end{figure}
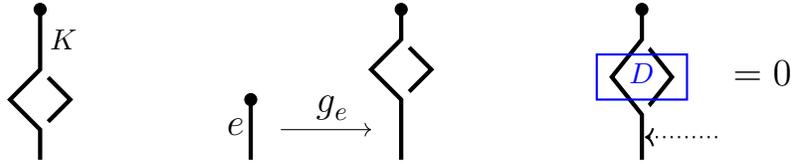

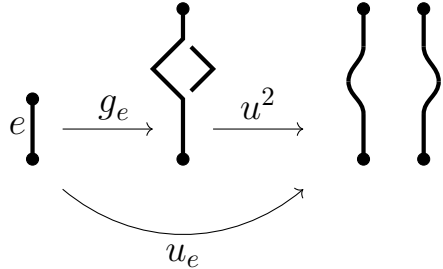
\begin{figure}
\begin{tikzpicture}[scale=.4]

\begin{scope}[xshift=7cm]
\draw[ultra thick] (8,-4)--(8,-2);
\filldraw[fill=black, draw=black] (8,-2) circle (2mm) ;
\filldraw[fill=black, draw=black] (8,-4) circle (2mm) ;

\draw[ultra thick](13,0)--(14,-1)--(13,-2);
\draw[line width= 3 mm, white](13,1)--(13,0)--(12,-1)--(13,-2)--(13,-4);
\draw[ultra thick](13,1)--(13,0)--(12,-1)--(13,-2)--(13,-4);
\filldraw[fill=black, draw=black] (13,1) circle (2mm) ;
\filldraw[fill=black, draw=black] (13,-4) circle (2mm) ;

\draw[->] (9,-3)--node[above]{\Large $g$}(12,-3);
\node[] at (11,-2.5) {\small $e$};
\node[] at (7.5,-2.9) {\Large $e$};
\end{scope}

\begin{scope}[xshift=12cm]
\draw[->] (9,-3)--node[above]{\Large $u^2$}(12,-3);
\filldraw[fill=black, draw=black] (14,1) circle (2mm) ;
\filldraw[fill=black, draw=black] (14,-4) circle (2mm) ;
\draw[ultra thick] (14,-4)--(14,-2.5)(14,-.25)--(14,1);
\draw[ultra thick]  (14,-2.5) to [out=90, in=270] (13.5,-1.42);
\draw[ultra thick]  (13.5,-1.42) to [out=90, in=270] (14,-.25);

\filldraw[fill=black, draw=black] (16,1) circle (2mm) ;
\filldraw[fill=black, draw=black] (16,-4) circle (2mm) ;
\draw[ultra thick] (16,-4)--(16,-2.5)(16,-.25)--(16,1);
\draw[ultra thick]  (16,-2.5) to [out=90, in=270] (16.5,-1.42);
\draw[ultra thick]  (16.5,-1.42) to [out=90, in=270] (16,-.25);

\end{scope}

\draw[->](16,-5) to  [out=320, in=220] (24,-5);
\node[below]at (20,-6.5){\Large $u_e$};
\end{tikzpicture}
\captionof{figure}{Sphere unzip is the composition of the cactus grafting and two disc unzips.}\label{fig: cactuscomp}
\end{figure}

\subsection{The diagrammatic statement}

Applying the homomorphic expansion $Z$ to the topological statement of Section \ref{eqfoams} gives rise to the following equation:

  \[ Z\left(
  \begin{tikzpicture}[scale=.17,, baseline=-1.3mm]
\draw[thick](0,-1)--(0,1);
\draw[red](-1,0)--(-.4,0);
\draw[red] (.4,0)--(.7,0);
\filldraw[fill=black, draw=black] (0,1) circle (2mm) ;
\filldraw[fill=black, draw=black] (0,-1) circle (2mm) ;
\draw[thick](1,-1)--(1,1);
\draw[red](1.4,0)--(2,0);
\filldraw[fill=black, draw=black] (1,1) circle (2mm) ;
\filldraw[fill=black, draw=black] (1,-1) circle (2mm) ;
\end{tikzpicture} \right)=Z\left(
  \begin{tikzpicture}[scale=.17,, baseline=-1.3mm]
\draw[thick](0,-1)--(0,1);
\draw[red](-1,0)--(-.4,0);
\draw[red] (.4,0)--(1,0);
\filldraw[fill=black, draw=black] (0,1) circle (2mm) ;
\filldraw[fill=black, draw=black] (0,-1) circle (2mm) ;
\end{tikzpicture} \right)\# Z\left(  \begin{tikzpicture}[scale=.17,, baseline=-1.3mm]
\draw[thick](0,-1)--(0,1);
\draw[red](-1,0)--(-.4,0);
\draw[red] (.4,0)--(1,0);
\filldraw[fill=black, draw=black] (0,1) circle (2mm) ;
\filldraw[fill=black, draw=black] (0,-1) circle (2mm) ;
\end{tikzpicture} \right) = u Z \left(
  \begin{tikzpicture}[scale=.17,, baseline=-1.3mm]
\draw[thick](0,-1)--(0,1);
\draw[red](-1,0)--(-.4,0);
\draw[red] (.4,0)--(1,0);
\filldraw[fill=black, draw=black] (0,1) circle (2mm) ;
\filldraw[fill=black, draw=black] (0,-1) circle (2mm) ;
\end{tikzpicture} \right) \]

 Using the notation of Figure~\ref{fig:ValuesOfZ}, we can compute each term of this individually to obtain the final form of the diagrammatic statement, as shown in Figure~\ref{fig:comps}.

 \begin{figure}

 \begin{tikzpicture}[scale=.5]
\draw[thick](-9.5,-1) rectangle (9.5,10);
\node[]at (0,9){\Large \textbf{Diagramatic Statement}};
\node[]at (-.5,5){\Huge $=$};
\begin{scope}[xshift=5cm,yshift=5cm]
\draw[ultra thick] (-1,-2.75)--(-1,2.75);
\filldraw[fill=black, draw=black] (-1,-2.75) circle (2mm) ;
\filldraw[fill=black, draw=black] (-1,2.75) circle (2mm) ;
\draw[red,->](2,-2.75)--(2,2.75);

\begin{scope}[xshift=-2cm]
\draw[ultra thick] (-1,-2.75)--(-1,2.75);
\filldraw[fill=black, draw=black] (-1,-2.75) circle (2mm) ;
\filldraw[fill=black, draw=black] (-1,2.75) circle (2mm) ;
\draw[thick,dotted ,->] (-.7,0)--(4,0);
\node[] at (4,0) {\hspace{12mm}$e^{a_1+a_2}$};

\begin{scope}[yshift=-.5cm]
\draw[fill=white] (-1.6,-1) rectangle (1.5,-1.8);
\node[]at(0,-1.4) {\Small $u(C)$};
\end{scope}

\begin{scope}[yshift=3.25cm]
\draw[fill=white] (-1.6,-1) rectangle (1.5,-1.8);
\node[]at(0,-1.4) {\Small $u(C)$};
\end{scope}

\begin{scope}[yshift=1.4cm]
\draw[fill=white] (-1.6,-1) rectangle (1.5,-1.8);
\node[]at(0,-1.4) {\Small $+$};
\end{scope}\end{scope}

\end{scope}

\begin{scope}[xshift=-5cm,yshift=5cm]
\draw[ultra thick] (-1,-2.75)--(-1,2.75);
\filldraw[fill=black, draw=black] (-1,-2.75) circle (2mm) ;
\filldraw[fill=black, draw=black] (-1,2.75) circle (2mm) ;
\draw[dotted,thick ,->] (-.7,.75)--node[above]{$e^{a_2}$}(2,.75);
\draw[red,->](2,-2.75)--(2,2.75);

\begin{scope}[yshift=-.5cm]]
\draw[fill=white] (-1.6,-1.1) rectangle (-.4,-1.8);
\node[]at(-1,-1.4) {\Small $C$};
\end{scope}

\begin{scope}[yshift=3.5cm]
\draw[fill=white] (-1.6,-1.1) rectangle (-.4,-1.8);
\node[]at(-1,-1.4) {\Small $C$};
\end{scope}

\begin{scope}[xshift=-2cm]
\draw[ultra thick] (-1,-2.75)--(-1,2.75);
\filldraw[fill=black, draw=black] (-1,-2.75) circle (2mm) ;
\filldraw[fill=black, draw=black] (-1,2.75) circle (2mm) ;
\draw[thick,dotted ,->] (-.7,-.65)--node[below]{\hspace{10mm}$e^{a_1}$}(4,-.65);

\begin{scope}[yshift=-.5cm]]
\draw[fill=white] (-1.6,-1.1) rectangle (-.4,-1.8);
\node[]at(-1,-1.4) {\Small $C$};
\end{scope}

\begin{scope}[yshift=3.5cm]
\draw[fill=white] (-1.6,-1.1) rectangle (-.4,-1.8);
\node[]at(-1,-1.4) {\Small $C$};
\end{scope}
\end{scope}

\end{scope}
\end{tikzpicture}

\vspace{-1.75cm}
\[\hspace{-1.5cm}Z\left(
  \begin{tikzpicture}[scale=.17,, baseline=-1.3mm]
\draw[thick](0,-1)--(0,1);
\draw[red](-1,0)--(-.15,0);
\draw[red] (.15,0)--(1,0);
\filldraw[fill=black, draw=black] (0,1) circle (2mm) ;
\filldraw[fill=black, draw=black] (0,-1) circle (2mm) ;
\end{tikzpicture} \right)\# Z\left(  \begin{tikzpicture}[scale=.17,, baseline=-1.3mm]
\draw[thick](0,-1)--(0,1);
\draw[red](-1,0)--(-.15,0);
\draw[red] (.15,0)--(1,0);
\filldraw[fill=black, draw=black] (0,1) circle (2mm) ;
\filldraw[fill=black, draw=black] (0,-1) circle (2mm) ;
\end{tikzpicture} \right)\hspace{.75cm}
= \hspace{.5cm} u Z \left(
  \begin{tikzpicture}[scale=.17,, baseline=-1.3mm]
\draw[thick](0,-1)--(0,1);
\draw[red](-1,0)--(-.15,0);
\draw[red] (.15,0)--(1,0);
\filldraw[fill=black, draw=black] (0,1) circle (2mm) ;
\filldraw[fill=black, draw=black] (0,-1) circle (2mm) ;
\end{tikzpicture}  \right)\]

\captionof{figure}{The Diagrammatic Statement of ``$1+1=2$'' in $\calA$. Since tails commute, two caps on a strand can be combined into $C^2$, and the two unzipped caps can be combined as $u(C)^2=u(C^2)$.}\label{fig:comps}
\end{figure}
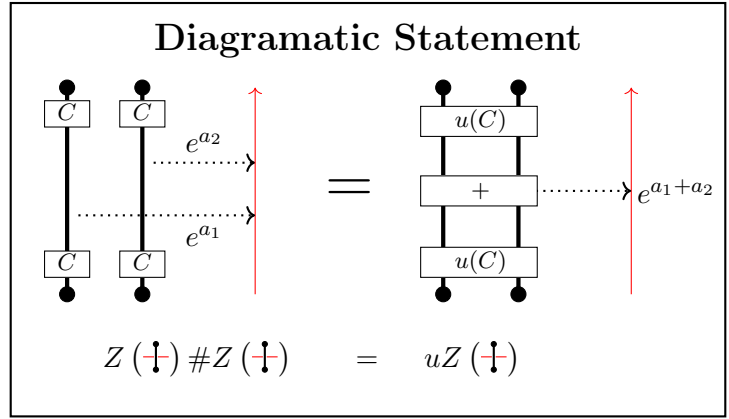

 \section{Understanding the Tensor Statement}

 Ultimately we aim to give a solution of the Duflo problem, which concerns finite dimensional Lie algebras. From here on, we fix a finite dimensional Lie algebra $\g$ over a field $\mathbb K$ of characteristic zero. Let $I\g$ denote the {\em double} of $\g$, that is, the semidirect product $\g^* \rtimes \g$, where $\g^*$ is taken to be abelian, and $\g$ acts on $\g^*$ by the coadjoint action. In formulae:

$$ I\g=\{(\varphi, x): \varphi\in \g^*, x \in \g\}, $$
  \begin{equation}\label{eq:Ig}
 [(\varphi_1,x_1),(\varphi_2, x_2)]=(x_1\cdot\varphi_2-x_2\cdot\varphi_1,[x_1,x_2]).
 \end{equation}

 We will define a \textit{tensor interpretation} map
 $$T:\mathcal A^w(\uparrow^n)\to (U(I\g)^{\otimes n})^{\wedge}, $$
where $A^w(\uparrow^n)$ denotes arrow diagrams on $n$ thick black strands, modulo the three $\aSTU$ relations only\footnote{It would be called $\calA(\uparrow^n)$ if we also imposed RI.}. $U(I\g)$ is the universal enveloping algebra of $I\g$, and $^\wedge$ denotes the degree completion, where elements of $\g^*$ are defined to be degree 1, and elements of $\g$ are degree zero; with this grading $T$ is degree-preserving.
We will show that when $n=k_1+k_2$, $T$ descends to a map
 $$T: \calAt\rightarrow (S(\g^*)_{\g}^{\otimes {k_1}}\otimes U(\g)^{\otimes {k_2}})^\wedge.$$

 Here $\calAt$ is the space of arrow diagrams on the skeleton of $k_1$ spheres and $k_2$  strings -- see Figure~\ref{WheelExample} for examples -- modulo all three $\aSTU$ relations, as well as the $CP$ relation at both caps of each capped strand,  as in Section~\ref{sec:arrows}. Note that the $RI$ relation is vacuous on red strings, and it follows from the CP relation on the twice-capped strands: short arrows can be commuted to the cap they point towards, hence they vanish by $CP$.

 As for the target, $S(\g^*)$ denotes the symmetric algebra of the linear dual of $\g$, and the subscript $\g$ denotes co-invariants under the co-adjoint action of $\g$: that is, the quotient by the image of the co-adjoint action. $U(\g)$ is the universal enveloping algebra of $\g$, and $^\wedge$ denotes the degree completion where elements of $\g^*$ are degree one and elements of $\g$ are degree zero, as before.

 \begin{figure}
\begin{tikzpicture}[scale=.8]
\draw[ultra thick](-.5,-.25)--(-.5,2.5);
\draw[fill=black] (-.5,-.25) circle (.1cm);
\draw[fill=black] (-.5,2.5) circle (.1cm);

\draw[ ->,thick,dotted ] (-.5,.5)--(.5,.5);
\draw[->,thick,dotted ] (-.5,1.25)--(0.15,1.25);
\draw[->,thick,dotted ] (-.5,2)--(.5,2);

\draw[->,thick,dotted ] (.5,.5)to [out=150, in=270] (.15,1.2);
\draw[->,thick,dotted ] (.15,1.3)to [out=90, in=190] (.5,1.95);
\draw[->,thick,dotted ] (.5,2) to [out=0, in=360] (.55,.5);
\end{tikzpicture} \hspace{2.5cm}
\begin{tikzpicture} [scale=.8]
\draw[ultra thick](-.5,-.25)--(-.5,2.5);
\draw[fill=black] (-.5,-.25) circle (.1cm);
\draw[fill=black] (-.5,2.5) circle (.1cm);
\draw[red,->](2,-.25)--(2,2.5);
\draw[red,->](2.3,-.25)--(2.3,2.5);

\begin{scope}[xshift=-1cm]
\draw[ultra thick](-.5,-.25)--(-.5,2.5);
\draw[fill=black] (-.5,-.25) circle (.1cm);
\draw[fill=black] (-.5,2.5) circle (.1cm);
\end{scope}

\draw[ ->,thick,dotted ] (-1.5,0)--(.5,0);
\draw[->,thick,dotted ] (-.5,.35)--(0.15,.35);
\draw[->,thick,dotted ] (-1.5,.7)--(.15,.7);
\draw[->,thick,dotted ] (-.5,1.05)--(.5,1.05);

\draw[->,thick,dotted ] (.5,0)to [out=150, in=270] (.15,.35);
\draw[->,thick,dotted ] (.15,.35)--(.15,.7);
\draw[->,thick,dotted ] (.15,.7) to [out=70, in=220] (.55,1.05);
\draw[->,thick,dotted ] (.55,1.05) to [out=0, in=0] (.5,0);

\draw[->,thick,dotted ] (-.5,1.5)--(2,1.5);
\draw[->,thick,dotted ] (-.5,2)--(2,2);
\node[]at (2.7,1.15){\Large $+$};
\end{tikzpicture}  \begin{tikzpicture} [scale=.8]
\draw[ultra thick](-.5,-.25)--(-.5,2.5);
\draw[fill=black] (-.5,-.25) circle (.1cm);
\draw[fill=black] (-.5,2.5) circle (.1cm);
\draw[red,->](2,-.25)--(2,2.5);
\draw[red,->](2.3,-.25)--(2.3,2.5);

\begin{scope}[xshift=1cm]
\draw[ultra thick](-.5,-.25)--(-.5,2.5);
\draw[fill=black] (-.5,-.25) circle (.1cm);
\draw[fill=black] (-.5,2.5) circle (.1cm);
\end{scope}

\draw[->,thick,dotted ] (-.5,.8)--(2.3,.8);

\draw[->,thick,dotted ] (.5,.5)--(2,.5);
\draw[->,thick,dotted ] (.5,.25)--(2,.25);
\draw[->,thick,dotted ] (.5,0)--(2,0);

\begin{scope}[yshift=.5cm]
\draw[->,thick, dotted ] (1.5,1.25)--(2,1.25);
\draw[->,thick,dotted ] (-.5,1.6)--(1.5,1.25);
\draw[->,thick,dotted ] (.5,.75)--(1.5,1.25);
\end{scope}

\end{tikzpicture}
\captionof{figure}{ An wheel with three spokes in $\protect \calAtSph$, and an element in ${ \protect \calAtFour }$.}\label{WheelExample}
\end{figure}

 \subsection{The Tensor Intpretation Map}\label{sec:mapT}
 The idea in the construction of $T:\mathcal A^w(\uparrow^n)\to (U(I\g)^{\otimes n})^{\wedge}$
  is that trivalent arrow vertices ``represent'' the Lie bracket in $\g$, and the relations in $\calA(\uparrow^n)$ translate to Lie algebra axioms.

 Denote the structure tensor of the Lie bracket of $\g$ by $[,]_\g\in\g^*\otimes\g^*\otimes \g$. Given a basis  $\{b_1,\cdots, b_m\}$ for $\g$ and the dual basis $\{b_1^*,\cdots, b_m^*\}$ for $\g^*$, write $[b_i,b_j]=\sum_{k=1}^m c_{ij}^k b_k$ for structure constants $c_{ij}^k \in \mathbb K$. Then
 $$[,]_{\g}=\sum_{i=1}^n c_{ij}^k \, b_i^*\otimes b_j^*\otimes b_k.$$
 Similarly, let $\id \in\g^*\otimes \g$ denote the identity tensor, given by $\id=\sum_{i=1}^m b_i^*\otimes b_i$.



  For an arrow diagram $D$ we first define $T(D)$ in the tensor algebra $\mathcal T (I\g)^{\otimes n}$ as follows, and as shown in Figure~\ref{StepsOfT}:
 \begin{enumerate}
\item  Place a copy of $\id \in \g^* \otimes \g$ on every single arrow, with $\g^*$ at the arrow tail, and $\g$ at the arrow head.

\item Place the structure tensor of the bracket $[,]_\g$ on all trivalent arrow vertices, with $\g^*$ components on the incoming arrows and the $\g$ component at the outgoing arrow.

\item Arrows which connect two trivalent vertices now have an element of $\g^*$ meeting an element of $\g$. Contract these by evaluating the element of $\g^*$ on the element of $\g$ to get a coefficient in $\mathbb K$. Multiply the constants together.

\item What remains is a linear combination of diagrams with elements of $\g^*$ and $\g$ along the strands. Multiplying these in $\mathcal T (I\g)$ along the orientation of the strands produces an element in  $\mathcal T({I\g}^{\otimes {n}})$.

 \end{enumerate}

Given the bases $\{b_1,...,b_m\}$ for $\g$ and $\{b_1^*,...,b_m^*\}$ for $\g^*$, a simple way to compute the value of $T$ on a given diagram $D$ is to some over all ways of labelling each arrow with an index $1,...,m$, form the corresponding product in $\mathcal T (I\g)^{\otimes n}$ taking a factor $b_i$ for each arrow tail labelled $i$ and $b_j^*$ for each arrow head labelled $j$; with the coefficient given by the product of $c_{rs}^t$ for each arrow vertex with incoming arrows labelled $r$ and $s$ and outgoing arrow labelled $t$. See Figures~\ref{ThreeSpokeWheel} and~\ref{ExampleT} for sample computations of $T$ for two arrow diagrams.

\begin{figure}
\[  \begin{tikzpicture} [scale=.8,baseline=.8cm]
\draw[ultra thick,](0,0)--(0,2);
\draw[thick,->](0,0)--(0,2.05);
\draw[ultra thick,](2,0)--(2,2);
\draw[ thick,->](2,0)--(2,2.05);
\draw[->,dotted, thick](0,1)--node[above]{``id''}(2,1);
\node[left] at (0,1){$\g^*$};
\node[right] at (2,1){$\g$};
\end{tikzpicture} := \sum_{i=1}^m \begin{tikzpicture} [scale=.8,baseline=.8cm]
\draw[ultra thick,](0,0)--(0,2);
\draw[thick,->](0,0)--(0,2.05);
\draw[ultra thick,](2,0)--(2,2);
\draw[ thick,->](2,0)--(2,2.05);
\draw[->,dotted,thick ](0,1)--(2,1);
\node[left] at (0,1){$b_i^*$};
\node[right] at (2,1){$b_i$};
\end{tikzpicture} \]

\[  \begin{tikzpicture} [scale=.8,baseline=-.1cm]
\draw[->,dotted,thick ](-1,1)--(-.02,.02);
\draw[->,dotted,thick ](-1,-1)--(-.02,-.02);
\draw[->,dotted,thick ](0,0)--(1.25,0);
\node[] at(.25,.5) {$[,]_\g$};

\end{tikzpicture}\hspace{2mm} := \sum_{i,j,k=1}^m c_{ij}^k  \begin{tikzpicture} [scale=.8,baseline=-.1cm]
\draw[->,dotted,thick ](-1,1)--(-.02,.02);
\draw[->,dotted,thick ](-1,-1)--(-.02,-.02);
\draw[->,dotted,thick ](0,0)--(1.25,0);
\node[left] at(-1,1) {$b_i^*$};
\node[left] at(-1,-1) {$b_j^*$};
\node[right] at(1.25,0) {$g_k$};

\end{tikzpicture}
\]
\captionof{figure}{Steps (1) and (2) in computing $T$.}\label{StepsOfT}
\end{figure}

\begin{figure}
\[\sum_{e,f,h,i,j,k=1}^m  c_{ke}^f\cdot c_{fh}^i \cdot c_{ij}^k
\begin{tikzpicture}[scale=.8,baseline=0cm]
\draw[ultra thick](-1.5,-2)--(-1.5,2);
\draw[ thick, ->](-1.5,-2)--(-1.5,2.05);

\draw[ ->,dotted,thick ] (-1.5,-1.5)--node[below]{\Small $e$}(.5,-1.5);
\draw[->,dotted,thick ] (-1.5,0)--node[below]{\Small $h$}(0,0);
\draw[->,dotted,thick ] (-1.5,1.5)--node[above]{\Small $j$}(.5,1.5);

\draw[->,dotted,thick ] (.5,-1.5)to [out=150, in=270] (0,-.03);
\draw[->,dotted,thick ] (0,.03)to [out=90, in=210] (.5,1.43);
\draw[->,dotted,thick ] (.5,1.5) to [out=0, in=90] node[right]{\Small$k$} (1.4,-.05);
\draw[->,dotted,thick ] (1.4,-.05) to [out=270, in=360]node[right]{} (.5,-1.5);

\draw[->,dotted,thick ] (.05,-.8);
\draw[->,dotted,thick ] (.05,.8);

\node[] at (.3,-.45){\Small $f$};
\node[] at (.4,1){\Small $i$};
\end{tikzpicture}
= \sum_{c,e,f,i,j,k=1}^m c_{ke}^f\cdot c_{fh}^i \cdot c_{ij}^k \cdot b_e^* \cdot b_h^* \cdot b_j^*
 \in \mathcal T(I\g)
\]

\captionof{figure}{Example computation of $T$ for a wheel with three spokes.}\label{ThreeSpokeWheel}
\end{figure}

 \begin{figure}
\[
\begin{tikzpicture}[scale=.75,baseline=.5cm]

\draw[ultra thick](-.5,-.5)--(-.5,2);
\draw[thick,->](-.5,-.5)--(-.5,2.05);
\draw[ultra thick](2,-.5)--(2,2);
\draw[ thick,->](2,-.5)--(2,2.05);

\draw[ ->,dotted,thick ] (-.5,0)--(.55,0);
\draw[->,dotted,thick ] (-.5,1.05)--(.55,1.05);

\draw[,dotted,thick ] (.5,0)to [out=150, in=270] (.15,.35);
\draw[->,dotted,thick ] (.15,.35)--(.15,.7);
\draw[,dotted,thick ] (.15,.7) to [out=70, in=180] (.55,1.05);
\draw[,dotted,thick ] (.55,1.05) to [out=0, in=0] (.5,0);
\draw[dotted,thick,->] (.83,.55)--(.83,.54);

\draw[->,dotted,thick ] (-.5,1.5)--(2,1.5);
\draw[|-|](-.6,-1)--node[below]{$D$}(2.1,-1);
\end{tikzpicture} \rightsquigarrow
\!\!\!\!\!  \sum_{f,i,j,k,l=1}^m
 \!\!\!\!\!\! c_{kf}^i \cdot c_{ij}^k
\begin{tikzpicture}[scale=.75,baseline=.5cm]

\draw[ultra thick](-.5,-.5)--(-.5,2);
\draw[thick,->](-.5,-.5)--(-.5,2.05);
\draw[ultra thick](2,-.5)--(2,2);
\draw[ thick,->](2,-.5)--(2,2.05);

\draw[ ->,dotted,thick ] (-.5,0)--(.55,0);
\node[below]at (0,.05) {\tiny $j$};
\draw[->,dotted,thick ] (-.5,1.05)--(.55,1.05);
\node[above] at (0,1.0) {\tiny $f$};

\draw[,dotted,thick ] (.5,0)to [out=150, in=270] (.15,.35);
\draw[->,dotted,thick ] (.15,.35)--(.15,.7);
\draw[,dotted,thick ] (.15,.7) to [out=70, in=180] (.55,1.05);
\draw[,dotted,thick ] (.55,1.05) to [out=0, in=0] node[right]{\Small $k$}(.5,0);
\draw[dotted,thick,->] (.83,.55)--(.83,.54);


\node[] at (-.2,.5) {\Small $i$};

\draw[->,dotted,thick ] (-.5,1.6)--node[above]{\Small $l$}(2,1.6);
\end{tikzpicture}
\rightsquigarrow
\!\!\!\!\! \sum_{f,i,j,k,l=1}^m \!\!\! c_{kf}^i \cdot c_{ij}^k\cdot (b_f^*b_j^*b_l^* \otimes b_l) \in \mathcal T(I\g) \otimes \mathcal T(I\g)
\]

\caption{An example for computing T.}\label{ExampleT}
\end{figure}

\begin{lem}\label{lem:welldef}
$T$ descends to a well defined map $\mathcal A^w (\uparrow^n) \rightarrow (U(I\g)^{\otimes n})^\wedge$.
\end{lem}

 \begin{proof}
 We need to check that relations in $\mathcal A^w(\uparrow^n)$ are mapped to relations in $(U(I\g)^{\otimes n})^\wedge$. This is indeed the case: $\aSTU_1$ and $\aSTU_2$
 are mapped to the relations $[x_i,x_j]=x_ix_j-x_jx_i$, and $[\varphi_i,x_j]=\varphi_ix_j-x_j\varphi_i$. $\aSTU_3=TC$ is the fact that $\g^*$ is abelian. It is also easy to check that the $T$ is degree preserving, so the completions on both sides agree.
 \end{proof}

 \begin{prop} When $k_1+k_2=n,$ $T$ further descends to a well defined map
 $$T: \calAt\rightarrow \left(S(\g^*)_{\g}^{\otimes {k_1}}\otimes U(\g)^{\otimes {k_2}}\right)^\wedge.$$
 \end{prop}

 \begin{proof}
 Before we prove the statement, we establish the following crucial Lie theory fact: $$\g \backslash U(I\g) / \g \cong (S\g^{*})_{\g}.$$ The notation on the left hand side means factoring out by the multiplication action of $\g$ on the left and on the right. First of all, as
 vector spaces $U(I\g)\cong U(\g^*)\otimes U(\g)\cong S(g^*)\otimes U(\g)$.
 The Poincare--Birkhoff--Witt (PBW) Theorem implies that this isomorphism also holds as $\g$-bimodules, that is, the left multiplication action of $\g$ on $U(I\g)$ is the
 coadjoint action on $S(\g^*)$, and the right multiplication action is simply right multiplication in $U(\g)$. Hence, factoring out by right multiplication gives $S(\g^*)$ and then modding out by left multiplication yields $(S\g^{*})_{\g}$, as claimed.

 We now re-phrasing the statement as a commutative diagram:
 \[
 \begin{tikzcd}
  \mathcal A(\uparrow^n) \arrow[r, "T"] \arrow[d, two heads,"\pi_1"]
    & \left(U(I\g)^{\otimes k_1}\otimes U(I\g)^{\otimes k_2}\right)^\wedge \arrow[d, two heads, "\pi_2"] \\
  \calAt \arrow[r, dashed, "\exists T"]
& \left(S(\g^*)_{\g}^{\otimes {k_1}}\otimes U(\g)^{\otimes {k_2}}\right)^\wedge \end{tikzcd}
\]
 Here the projection $\pi_1$ is imposing two cap relations on each of the first $k_{1}$ strands, and killing all arrow diagrams with any arrow tails on strands $k_{1}+1$ through $n$. The map $\pi_2$ is the projection $U(I\g) \to \g \backslash U(I\g) / \g$ in the first $k_2$ tensor factors, and the projection given by setting $\varphi=0$ for any $\varphi \in \g^*$ in the last $k_2$ tensor factors.
One then defines the bottom horizontal $T$ map in the obvious way: taking any pre-image of $\pi_1$, and applying $T$ followed by $\pi_2$. For this to be well-defined, we need to show that any element in the kernel of $\pi_1$ is killed by the composition of $\pi_2\circ T$.

On the first $k_{1}$ strands, the cap relations assert that an arrow head at the very beginning or the very end of the strand is zero. The map $T$ assigns elements of $\g$ to arrow heads. Hence, the projection $\pi_2$, which factors out by the multiplication action of $\g$ on $U(I\g)$ both on the left and on the right makes the bottom map well defined.

 On strands $k_{1} +1$ through $n$, $\pi_1$ kills all arrow tails. Under the tensor interpretation map $T$, arrow tails ending on a given strand translate to elements of $\g^{*}$ in a product in the corresponding $U(I\g)$ tensor factor. Hence the quotient map which sets elements of $\g^{*}$ to be zero makes the bottom horizontal map well defined, and this is $\pi_2$.
  \end{proof}

 The following Proposition will play a crucial role later; we present it here as it is based on a similar principle as the proof above:

  \begin{prop}\label{prop:invariance} The image of $T$ is $\g$-invariant, where $\g$ acts via the adjoint action on
each of the $U\g$ factors (that is, sum over acting on each one):
$T(D) \in ({S}\g^*)_\g^{\otimes k_1}\otimes \left(({U}\g)^{\otimes k_2}\right)^{\g}$ for any diagram $D \in \calAt$.
\end{prop}

\begin{proof}
To give a short argument, the structure tensors of the identity and the Lie bracket are invariant elements, which are composed in an invariant way when computing the
image under $T$ of an arrow diagram, hence the result is $\g$-invariant.

A more hands-on proof uses the {\em head invariance} property of arrow diagrams:  the relevant incarnation is shown in
Figure~\ref{fig:headinvariance}; the property in general is discussed in \cite[Remark 3.14]{WKO2}. In words, the sum over all thin red strands of attaching an additional arrow head at the bottom of the strand gives the same result as the sum over all thin red strands of attaching the arrow head at the top.

To prove the Proposition, we apply $T$ on both sides of the equality of Figure~\ref{fig:headinvariance}. Attach an additional arrow head at the bottom of the $i$-th red strand of an arrow diagram $D$, where the arrow tail lies on an additional strand as in Figure~\ref{fig:headinvariance}; call this new diagram $D_{i}$. Compare $T(D_{i})$ with $T(D)$:
$$T(D_{i})=\sum_{j=1}^{m}b_{j}^{*}\otimes (b_{j}\times_{i}T(D)),$$
where $\times_{i}$ denotes multiplying on the left in the $i$-th tensor factor.

Similarly, let $D^{i}$ denote the arrow diagram obtained from $D$ by attaching an arrow head at the top of the $i$-th red strand.
$$T(D^{i})=\sum_{j=1}^{m}b_{j}^{*}\otimes (T(D)\times_{i}b_{j}),$$
where multiplication is now on the right. Hence,
$$T(D_{i})-T(D^{i})=\sum_{j=1}^{m}b_{j}^{*}\otimes (b_{j}\cdot_{i}T(D)),$$
where $\cdot_{i}$ denotes the adjoint action on the $i$-th $U\g$ tensor factor. Due to the head invariance property we have
$$0=\sum_{i=1}^{k_{2}}T(D_{i})-T(D^{i})=\sum_{j=1}^{m}b_{j}^{*}\otimes \left(\sum_{i=1}^{k_{2}} b_{j}\cdot_{i}T(D)\right).$$
The right hand side is only zero if $\sum_{i=1}^{k_{2}} b_{j}\cdot_{i}T(D)=0$ for each $j=1...m,$ which is exactly the statement of the Proposition.
\end{proof}

 \begin{figure}
\[
 \begin{tikzpicture}[baseline=.5cm]

 \draw[ultra thick ] (-.5,-1)--(-.5,1.5);
  \draw[->, thick ] (-.5,-.5)--(-.5,1.55);

 \draw[ultra thick] (0,0)--(0,1);
  \draw[ultra thick] (.5,0)--(.5,1);
   \draw[ultra thick] (1.25,0)--(1.25,1);
   \draw[fill=black] (0,0) circle (.08cm);
    \draw[fill=black] (0,1) circle (.08cm);
       \draw[fill=black] (.5,0) circle (.08cm);
    \draw[fill=black] (.5,1) circle (.08cm);
     \draw[fill=black] (1.25,0) circle (.08cm);
    \draw[fill=black] (1.25,1) circle (.08cm);

 \draw[thick, red, ->] (1.5,0)--(1.5,1);
    \draw[thick,red,] (3,-.55)to [out=180, in=270](1.5,0);
     \draw[thick,red,] (1.5,1)to [out=90, in=180](3,1.55);

  \draw[thick, red, ->] (1.75,0)--(1.75,1);
     \draw[thick,red,] (3,-.4)to [out=180, in=270](1.75,0);
         \draw[thick,red,] (1.75,1)to [out=90, in=180](3,1.4);

   \draw[thick, red, ->] (2,.1)--(2,1);
   \draw[thick,red,] (3,-.25)to [out=180, in=270](2,.1);
    \draw[thick,red,] (2,1)to [out=90, in=180](3,1.25);

   \draw[thick, blue] (-.25,.25) rectangle (2.25, .75);
  \node[blue] at (.85,.5) {$D$};

  \draw[fill=white] (2.5,-.65) rectangle  (2.75,-.15);
  \draw[ dotted,thick ] (-.5,-.9)--(2.62,-.9)--(2.62,-.65);
   \draw[ dotted,thick ,->] (2.62,-.66)--(2.62,-.65);
   \node[] at (2.63,-.4) {\tiny +};

 \end{tikzpicture}\hspace{5mm}{\Huge{=}} \hspace{5mm}\begin{tikzpicture}[baseline=.5cm]

  \draw[ultra thick ] (-.5,-1)--(-.5,1.5);
  \draw[->, thick ] (-.5,-.5)--(-.5,1.55);

  \draw[ultra thick] (0,0)--(0,1);
  \draw[ultra thick] (.5,0)--(.5,1);
   \draw[ultra thick] (1.25,0)--(1.25,1);
   \draw[fill=black] (0,0) circle (.08cm);
    \draw[fill=black] (0,1) circle (.08cm);
       \draw[fill=black] (.5,0) circle (.08cm);
    \draw[fill=black] (.5,1) circle (.08cm);
     \draw[fill=black] (1.25,0) circle (.08cm);
    \draw[fill=black] (1.25,1) circle (.08cm);

        \draw[thick, red, ->] (1.5,0)--(1.5,1);
    \draw[thick,red,] (3,-.55)to [out=180, in=270](1.5,0);
     \draw[thick,red,] (1.5,1)to [out=90, in=180](3,1.55);

  \draw[thick, red, ->] (1.75,0)--(1.75,1);
     \draw[thick,red,] (3,-.4)to [out=180, in=270](1.75,0);
         \draw[thick,red,] (1.75,1)to [out=90, in=180](3,1.4);

   \draw[thick, red, ->] (2,.1)--(2,1);
   \draw[thick,red,] (3,-.25)to [out=180, in=270](2,.1);
    \draw[thick,red,] (2,1)to [out=90, in=180](3,1.25);


   \draw[thick, blue] (-.25,.25) rectangle (2.25, .75);
  \node[blue] at (.85,.5) {$D$};

  \draw[ dotted,thick ,->] (1.75,-.9)to [out=0, in=270](2.75,1.15);
  \draw[thick, dotted] (-.5,-.9)--(1.75,-.9);
    \draw[fill=white] (2.63,1.15) rectangle  (2.88,1.65);

    \node[] at (2.75, 1.4) {\tiny +};

 \end{tikzpicture}      \]

\caption{The {\em head invariance} property in $\protect \calAt$. }\label{fig:headinvariance}
 \end{figure}

 \subsection{The tensor statement}
 To obtain the tensor statement, we simply apply the map $T$ to the diagrammatic statement of Figure~\ref{fig:comps}:
 $$Z\left(
  \begin{tikzpicture}[scale=.17,, baseline=-1.3mm]
\draw[thick](0,-1)--(0,1);
\draw[red](-1,0)--(-.15,0);
\draw[red] (.15,0)--(1,0);
\filldraw[fill=black, draw=black] (0,1) circle (2mm) ;
\filldraw[fill=black, draw=black] (0,-1) circle (2mm) ;
\end{tikzpicture} \right)\# Z\left( \begin{tikzpicture}[scale=.17,, baseline=-1.3mm]
\draw[thick](0,-1)--(0,1);
\draw[red](-1,0)--(-.15,0);
\draw[red] (.15,0)--(1,0);
\filldraw[fill=black, draw=black] (0,1) circle (2mm) ;
\filldraw[fill=black, draw=black] (0,-1) circle (2mm) ;

\end{tikzpicture} \right)=  u^2 Z \left(
  \begin{tikzpicture}[scale=.17,, baseline=-1.3mm]
\draw[thick](0,-1)--(0,1);
\draw[red](-1,0)--(-.15,0);
\draw[red] (.15,0)--(1,0);
\filldraw[fill=black, draw=black] (0,1) circle (2mm) ;
\filldraw[fill=black, draw=black] (0,-1) circle (2mm) ;
\end{tikzpicture}  \right).$$
Since
$Z( \begin{tikzpicture}[scale=.17,, baseline=-1.3mm]
\draw[thick](0,-1)--(0,1);
\draw[red](-1,0)--(-.15,0);
\draw[red] (.15,0)--(1,0);
\filldraw[fill=black, draw=black] (0,1) circle (2mm) ;
\filldraw[fill=black, draw=black] (0,-1) circle (2mm) ;
\end{tikzpicture})$
is an element of $\calAone$,
 $T(Z(\!\raisebox{-1.3mm}{\threadedsphere}\!))=:\Upsilon \in ({S}(\g^*)_\g\otimes {U}(\g))^\wedge.$

 The connected sum operation in $\calA$ is concatenation along the red strings. Under the tensor interpretation map $T$, this translates to multiplication in ${U}(\g)$, while the ${S}(\g^*)$ components remain separate tensor factors. Hence, with notation explained below,

 $$T\left(Z(\!\raisebox{-1.3mm}{\threadedsphere}\!)\# Z(\!\raisebox{-1.3mm}{\threadedsphere}\!)\right)=\up^{13}\up^{23} \in ({S}(\g^*)_\g^{\otimes 2}\otimes {U}(\g))^\wedge.$$
 Here $\up^{13}:=\phi^{13}(\up)$, where $\phi^{13}(x\otimes y):=x\otimes 1\otimes y$, and $\up^{23}:=\phi^{23}(\up)$, where $\phi^{23}(x\otimes y):=1\otimes x\otimes y$.

 On the right side, the unzip operation sends an arrow ending on the unzipped strand to a sum of two arrows, one ending on each daughter strand. Under the tensor interpretation map $T$, this is sent to the Hopf algebra coproduct $\Delta$ of $\hat S(\g^*)_\g$ given by $\Delta(\varphi)=\varphi \otimes 1 + 1 \otimes \varphi$ for primitive elements $\varphi \in \g^*_\g$, as explained in Figure~\ref{fig:comultdiagram}.
 In other words, $T\circ u=\Delta\circ T$, and therefore

 $$T\left(u^2 Z \left(
  \!\raisebox{-1.3mm}{\threadedsphere}\!  \right)\right)=(\Delta \otimes 1)\up \in \left({S}(\g^*)_\g^{\otimes 2}\otimes {U}(\g)\right)^\wedge.$$

\begin{figure}
\begin{center}
\begin{tikzpicture}[scale=1, baseline=-1mm  ]
\draw[ultra thick](0,0)--(0,1);
\filldraw[fill=black, draw=black] (0,0) circle (.8mm) ;
\filldraw[fill=black, draw=black] (0,1) circle (.8mm) ;
\draw[dotted,thick ,->](0,.5)--(1,.5);
\node[left] at (0,.5) { $b_j^*$};
\draw[blue] (1,0) rectangle(1.5,1);

\draw[thick,->] (2.5,.5)--node[above]{$u_i$}(4,.5);

\begin{scope}[yshift=-3.5cm]
\draw[thick,->] (2.5,.5)--node[above]{$\Delta$}(4,.5);
\end{scope}

\begin{scope}[xshift=5cm]
\draw[ultra thick](0,0)--(0,1);
\filldraw[fill=black, draw=black] (0,0) circle (.8mm) ;
\filldraw[fill=black, draw=black] (0,1) circle (.8mm) ;
\draw[ultra thick](.5,0)--(.5,1);
\filldraw[fill=black, draw=black] (.5,0) circle (.8mm) ;
\filldraw[fill=black, draw=black] (.5,1) circle (.8mm) ;
\draw[dotted,thick ,->](0,.5)--(1.5,.5);
\node[left] at (0,.5) { $b_j^*$};
\draw[blue] (1.5,0) rectangle(2,1);
\node[]at (2.5,.5){$+$};
\end{scope}

\begin{scope}[xshift=8cm]
\draw[ultra thick](0,0)--(0,1);
\filldraw[fill=black, draw=black] (0,0) circle (.8mm) ;
\filldraw[fill=black, draw=black] (0,1) circle (.8mm) ;
\draw[ultra thick](.5,0)--(.5,1);
\filldraw[fill=black, draw=black] (.5,0) circle (.8mm) ;
\filldraw[fill=black, draw=black] (.5,1) circle (.8mm) ;
\draw[dotted,thick ,->](.5,.5)--(1.5,.5);
\node[right] at (0,.5) { $b_j^*$};
\draw[blue] (1.5,0) rectangle(2,1);
\end{scope}

\draw[thick,->] (.5,-.5)--node[right]{$T$}(.5,-2);
\draw[thick,->] (7.5,-.5)--node[right]{$T$}(7.5,-2);

\draw[blue](-1,-3)--(-.15,-3);
\node[] at (.5,-2.8){$\otimes b_j^*\otimes $};
\draw[blue](1,-3)--(1.75,-3);

\begin{scope}[xshift=7cm]
\draw[blue](-1.25,-3)--(-2,-3);
\node[] at (.5,-2.8){$\otimes (b_j^*\otimes 1 + 1\otimes b_j^* )\otimes $};
\draw[blue](2.25,-3)--(3,-3);
\end{scope}

\begin{scope}[xshift=7cm,yshift=-1cm]
\draw[blue](-1,-3)--(-.15,-3);
\node[] at (.5,-2.8){$\otimes \Delta b_j^*\otimes $};
\draw[blue](1,-3)--(1.75,-3);
\end{scope}

\end{tikzpicture}
\end{center}
\captionof{figure}{The tensor interpretation map intertwines unzip with co-multiplication.}\label{fig:comultdiagram}
\end{figure}

 In summary, we obtain the Tensor Statement of Figure~\ref{fig:tensorstate}. Note the similarity with the triangularity equation $R^{13}R^{23}=(\Delta\otimes 1)R$ of quasi-triangular Hopf algebras, yet this is not a triangularity equation as $S(\g)_\g$ and $U\g$ are not a dual pair.

 \begin{figure}[h!]
 \begin{tikzpicture}
 \draw[thick](-5.2,-1.2) rectangle (5.2,1.5);
\node[]at (0,1){\Large \textbf{Tensor Statement}};

 \node[]at (0,-.1){\Large $\displaystyle{\up^{13}\up^{23}=(\Delta\otimes 1)\up} \quad \text{ in }\quad ({S}(\g^*)_\g^{\otimes 2}\otimes {U}(\g))^\wedge$};


 \end{tikzpicture}
 \captionof{figure}{The tensor statement.}\label{fig:tensorstate}
 \end{figure}

 \section{The Duflo Isomorphism}

 There is a pairing $\g^*\times \g\rightarrow \mathbb{K}$ given by the evaluation. This extends to a pairing $S\g^*\times S\g\rightarrow \mathbb{K}$ in the following way. Given a monomial $\varphi_1\cdot ... \cdot \varphi_k \in S\g^*$ with $\varphi_i\in \g^*$ and a monomial of the same degree $x_1\cdot... \cdot x_k \in S\g$ with $x_j\in\g$,
\begin{equation}\label{eq:PairingOnS}
 \langle \varphi_1\cdot ... \cdot \varphi_k, x_1\cdot... \cdot x_k \rangle := \sum_{\sigma \in S_k} \varphi_1(x_{\sigma(1)})\cdot...\cdot\varphi_k(x_{\sigma(k)}),
 \end{equation}
 where the sum is over all permutations of the $k$ indices. Monomials pair as zero with any monomial of a different degree, and the pairing is then extended bilinearly.

 Alternatively, given a basis $\{b_1,...,b_m\}$ of $\g$ and dual basis $\{b_1^*,...,b_m^*\}$ of $\g^*$, $S\g$ and $S\g^*$ are spanned linearly by monomials in the basis
 elements $b_i$ and $b_j^*$ respectively. The monomial $(b_1^*)^{\alpha_1}\cdot...\cdot (b_m^*)^{\alpha_m}$ pairs as zero with every monomial in the basis vectors
 $\{b_i\}$ except $b_1^{\alpha_1}\cdot...\cdot b_m^{\alpha_m}$, and
 \begin{equation}\label{eq:basis}
\langle (b_1^*)^{\alpha_1}\cdot...\cdot (b_m^*)^{\alpha_m}, b_1^{\alpha_1}\cdot...\cdot b_m^{\alpha_m}\rangle=\prod_{i=1}^m\alpha_i!
\end{equation}

This descends to a pairing $(Sg^*)_\g\times (S\g)^\g\rightarrow \mathbb{K} $. For $\Pi \in (Sg^*)_\g$ and $P\in (S\g)^\g$ we will denote the value of this pairing by by $\langle\Pi,P\rangle$. Finally, one can extend to a pairing \linebreak $(Sg^*)_\g^{\otimes n} \times ((S\g)^\g)^{\otimes n}\rightarrow \mathbb{K}$ by simply multiplying the pairings of tensor factors. This satisfies the equality
\begin{equation}\label{eq:pairing}
\langle\Pi,PQ\rangle=\langle \Delta\Pi,P\otimes Q\rangle
\end{equation}
for any $P, Q \in (S\g)^\g$, where $\Delta$ is the co-product on $(Sg^*)_\g$ induced by the co-product on $Sg^*$.

Define the Duflo map $$\D: S(\g)^\g\rightarrow U(\g)^\g$$ by pairing with the first tensor factor of $\up\in ({S}(\g^*)_\g \otimes {U}(\g))^\wedge$, and by an abuse of notation, denote this by $\D(P)=\langle \up, P \rangle$. Note that although $\up$ lives in a degree completion, it is finite in each degree, and so only finitely many terms of $\up$ have non-zero pairings with any given $P\in S(\g)^\g$. Hence, $\D(P)\in U(\g)$. The fact that $\D(P)$ is $\g$-invariant is a direct consequence of Proposition~\ref{prop:invariance}.

\begin{thm}\label{thm:homom}
{\bf (The Duflo problem.)} The map $\D$ is an algebra homomorphism.
\end{thm}
\begin{proof}
After the work coomplished throughout the previous sections, this is now easy. By definition, $\D$ is linear. The multiplicativity of $\D$, on the other hand, follows directly from the Tensor Statement. Let $P,Q\in S(\g)^\g$, then

\begin{equation*}\D(PQ)=\langle \up, PQ \rangle \stackrel{1}{=} \langle (\Delta \otimes 1)\up, P\otimes Q \rangle \stackrel{2}{=} \langle \up^{13}\up^{23}, P\otimes Q \rangle
{=}\D(P)\D(Q).
\end{equation*}

Here Equality~$1$ is Equation~(\ref{eq:pairing}) above, and the second pairing is applied in the first two tensor factors of $(\Delta \otimes 1)\up$. Equality~$2$ is the Tensor Statement, and the pairing is still applied in the first two tensor factors of the first argument.
\end{proof}

\begin{prop}\label{prop:isom}
The map $\D$ is an algebra isomorphism.
\end{prop}

\begin{proof}
In light of Theorem~\ref{thm:homom} we only need to prove that $\D$ is bijective. Recall that $U\g$ is filtered by word length (aka the PBW filtration), and the PBW Theorem states that the associated graded algebra is isomorphic to $S\g$. We claim that $\D$ is a filtered map; this is true by inspection of $\up$ as below. We then prove that $\operatorname{gr}\D$ is the identity, hence $\D$ is bijective.

Recall that $Z(\!\!\raisebox{-1.2mm}{\threadedsphere}\!\!)$ consists of an exponential $e^a$ of an arrow $a$ from the capped startand to the red string, followed by a value $C^2$ on the twice-capped strand, as shown in Figure~\ref{fig:comps}. Hence, $$T(Z(\!\!\raisebox{-1.2mm}{\threadedsphere}\!\!))= T(e^a) \cdot (T(C^2)\otimes 1).$$
Since $T(a)= \iota$, where $\iota \in \g^*\otimes \g$ denotes the structure tensor of the identity morphism of $\g$, we have $T(e^a)=e^\iota.$
Recall from Formula (\ref{eq:C}) that $T(C^2)= 1 + \text{higher degree terms}$.

Now let $P\in S(\g)^\g$ be homogeneous of degree $d$. Observe that $\D(P)$ lives in filtered degree {\em at most} $d$, hence $\D$ is a filtered map.  In particular the only term of $\D(P)$ that doesn't belong to filtered degree $d-1$ arises from pairing with
$ \frac{1}{d!}\iota^d$.

Hence, the associated graded map $\operatorname{gr}\D(P)=\operatorname{gr}(P\mapsto\langle e^\iota, P \rangle)$, and by formula (\ref{eq:PairingOnS}) this is the identity of $S\g$, completing the proof.
\end{proof}

\subsection{The explicit formula} It remains to derive the explicit formula for the Duflo map. The starting point for this is the value $T\left(Z(\raisebox{-1.5mm}{\threadedsphere})\right)=e^\iota T(C^2)$, where $\iota$ is the identity tensor in $\g^*\otimes\g$. As discussed above as part of the proof of Proposition~\ref{prop:isom}, for any $P\in S\g$, the term of $\D(P)$ of highest filtered order arises from pairing with $e^\iota$.
We begin by understanding the map $\mathcal S$ given by this pairing, that is,
$ \mathcal S: S\g \to U\g $, $\mathcal S(P)=\langle e^\iota,P \rangle$.

Since $\mathcal S$ is linear, it is enough to consider the case where $P$ is a monomial of degree $d$ in the basis elements $b_i$, say $b_{i_1}b_{i_2}...b_{i_d}$, where $i_1\leq i_2\leq i_d$. The only term of $e^\iota$ that $P$ pairs with, then, is $\iota^d/d!$. Given that $\langle \iota, b_i \rangle=b_i$, then by the formula~(\ref{eq:PairingOnS}),
$$\left\langle\frac{\iota^d}{d!}, b_{i_1}b_{i_2}...b_{i_d} \right\rangle= \frac{1}{d!}\sum_{\sigma \in S_d} b_{\sigma(i_1)}...b_{\sigma(i_d)}. $$ In other words, $\mathcal S$ is simply the PBW symmetrization map.

Next, we need to understand what happens when $P$ is paired with the term $\frac{\iota^{d-r}}{(d-r)!} \cdot \delta_r$, where $\delta_r\in S\g^*$ is the degree $r$ term of $T(C^2)$, and the multiplication occurs in the first tensor factor. Although $T(C^2)$ is concentrated in even degrees, we compute $\langle \frac{\iota^{d-1}}{(d-1)!} \cdot b_i^*, P \rangle $ as an instructive example. Applying formula~(\ref{eq:PairingOnS}) in this case, we sum over pairing any one factor of $P$ with $b_i^*$, this factor disappears, and the map $\mathcal S$ is applied to the rest. This has the same effect as differentiating $P$ with respect to $b_i$ (denote this differential operator by $\partial_i$), then applying $\mathcal S$. In other words,
$$\left\langle \frac{\iota^{d-1}}{(d-1)!} \cdot b_i^*, P \right\rangle = \mathcal S (\partial_i P).$$

This conclusion extends easily to the case of pairing with  $\iota^{d-r}/(d-r)! \cdot \delta_r$. Namely, $\delta_r \in S\g^*$ can be written uniquely as a polynomial in the $b_i^*$. Let $D_r$ denote the differential operator on $S\g$ given by replacing each occurrence in $\delta_r$ of $b_i^*$  by $\partial_i$, for all i. Then,
$$\left\langle \frac{\iota^{d-r}}{(d-r)!} \cdot \delta_r, P \right\rangle= \mathcal{S} (D_r(P)).$$

Now recall that from Equation~\ref{eq:C},
$$C^2= \operatorname{exp}\left(\sum_{n=1}^\infty 2c_{2n}w_{2n}\right), \quad \text{ where } \quad
\sum_{n=1}^\infty c_{2n}x^{2n}=\frac{1}{4} \operatorname{log}\frac{\operatorname{sinh}{x/2}}{x/2}.$$
Here $w_{2n}$ denotes the wheel with $2n$ spokes. In Figure~\ref{ExampleT} we computed $T$ for a wheel with three spokes, and the result clearly generalises to any number of spokes. It remains to interpret this result in a more useful form. The value $T(w_{2n})$ lives in $S\g^*$, and elements of $S\g^*$ can be viewed as polynomial functions on $\g$. Brief inspection of $T(w_{2n})$ leads to the conclusion that for any $x\in \g$, the function $T(w_{2n})$ it is given by
$$T(w_{2n})(x)= \operatorname{Tr}(\operatorname{ad}_x)^{2n},$$
where ``$\operatorname{Tr}$'' denotes the trace and ``$\operatorname{ad}$'' stands for the adjoint representation.

Hence, $$T(C^2)= \operatorname{exp}\left(\sum_{n=1}^\infty 2c_{2n}\operatorname{Tr}(\operatorname{ad}_x)^{2n}\right)= \left(\frac{ \operatorname{sinh} \left(\frac{\operatorname{Tr} \operatorname{ad}_x}{2}\right)}{\frac{x}{2}}\right) ^{1/2}= \operatorname{det}^{1/2}\left( \frac{\operatorname{sinh}\left( \frac{\operatorname{ad}_x}{2}\right)}{\frac{ad_x}{2}} \right).$$
This agrees with the well-known formula for the Duflo isomorphism, see for example \cite{Duflo1}, or \cite[Section 1.1]{BLT}.

\end{document}